\documentclass[10pt,reqno]{amsart}
\usepackage{enumerate}
\usepackage{amsfonts,amssymb,amsmath,mathrsfs,graphicx}
\usepackage{float}
\usepackage{epsfig}
\usepackage{graphicx}
\usepackage{caption}
\usepackage{subcaption}
\usepackage{color}
\usepackage{setspace}
\DeclareMathOperator*{\supess}{ess\,sup}

\title[1-d diatomic Vlasov-Poision system]{one-dimensional diatomic Vlasov-Poisson system with oscillatory molecular bonds}

\author[Sun-Ho Choi]{Sun-Ho Choi}
\address[Sun-Ho Choi]{Department of Applied Mathematics and the Institute of Natural Sciences, Kyung Hee University, Yongin, 446-701, Republic of Korea}
\email{sunhochoi@khu.ac.kr}

\author[Seok-Bae Yun]{Seok-Bae Yun}
\address[Seok-Bae Yun]{Department of Mathematics, Sungkyunkwan University, Suwon 440-746, Republic of Korea}
\email{sbyun01@skku.edu}

\begin{document}
\newtheorem{theorem}{Theorem} [section]
\newtheorem{maintheorem}{Theorem}
\newtheorem{lemma}[theorem]{Lemma}
\newtheorem{proposition}[theorem]{Proposition}
\newtheorem{remark}[theorem]{Remark}
\newtheorem{example}{Example}
\newtheorem{exercise}{Exercise}
\newtheorem{definition}{Definition}[section]
\newtheorem{corollary}[theorem]{Corollary}


\newcommand{\noi}{\noindent}
\newcommand{\Z}{\mathbb{Z}}
\newcommand{\R}{\mathbb{R}}
\newcommand{\C}{\mathbb{C}}
\newcommand{\T}{\mathbb{T}}
\newcommand{\bul}{\bullet}

\newcommand{\N}{\mathcal{N}}
\newcommand{\RR}{\mathcal{R}}
\newcommand{\D}{\mathcal{D}}
\newcommand{\HH}{\mathcal{H}}

\newcommand{\al}{\alpha}
\newcommand{\dl}{\delta}
\newcommand{\Dl}{\Delta}
\newcommand{\eps}{\varepsilon}
\newcommand{\kk}{\kappa}
\newcommand{\g}{\gamma}
\newcommand{\G}{\Gamma}
\newcommand{\ld}{\lambda}
\newcommand{\Ld}{\Lambda}
\newcommand{\s}{\sigma}
\newcommand{\ft}{\widehat}
\newcommand{\wt}{\widetilde}
\newcommand{\cj}{\overline}
\newcommand{\dx}{\partial_x}
\newcommand{\dt}{\partial_t}
\newcommand{\dd}{\partial}
\newcommand{\invft}[1]{\overset{\vee}{#1}}
\newcommand{\lrarrow}{\leftrightarrow}
\newcommand{\embeds}{\hookrightarrow}
\newcommand{\LRA}{\Longrightarrow}
\newcommand{\LLA}{\Longleftarrow}

\newcommand{\wto}{\rightharpoonup}

\newcommand{\jb}[1]
{\langle #1 \rangle}

\newcommand{\kwon}[1]{{\color{blue} #1  }}


\renewcommand{\theequation}{\thesection.\arabic{equation}}
\renewcommand{\thetheorem}{\thesection.\arabic{theorem}}
\renewcommand{\thelemma}{\thesection.\arabic{lemma}}
\newcommand{\bbr}{\mathbb R}
\newcommand{\bbz}{\mathbb Z}
\newcommand{\bbn}{\mathbb N}
\newcommand{\bbs}{\mathbb S}
\newcommand{\bbp}{\mathbb P}
\newcommand{\ddiv}{\textrm{div}}
\newcommand{\bn}{\bf n}
\newcommand{\rr}[1]{\rho_{{#1}}}
\newcommand{\thh}{\theta}
\def\charf {\mbox{{\text 1}\kern-.24em {\text l}}}
\renewcommand{\arraystretch}{1.5}

\thanks{
\textbf{Acknowledgment.}S.-H. Choi was partially supported by NRF of Korea (no. 2017R1E1A1A03070692) and Korea Electric Power Corporation(Grant number: R18XA02). S.-B. Yun was supported by Samsung Science and Technology Foundation under Project Number SSTF-BA 1801-02.
}

\begin{abstract}In this paper, we derive a Vlasov type kinetic model for diatomic plasma in which each ion consists of two atoms bonded through an oscillatory intermolecular force given by a singular Hooke potential. We then consider the existence and uniqueness of the classical solution to the proposed model. A careful analysis of the oscillatory behavior of atoms near singular points of the Hooke potential is carried out to extend the local-in-time solution to the global solution. 
\end{abstract}

\maketitle


%
%

\section{Introduction}\label{intro}
\setcounter{equation}{0}


It is well known that the presence of strong magnetic fields causes significant changes in important transport coefficients for polyatomic molecules due to the precession of magnetic moments in the polyatomic molecules (i.e., the Senftleben-Beenakker effect \cite{K-R,T-C-T-B}). Considering that the control of plasma dynamics  is  often realized through the use of magnetic fields \cite{A-A}, this suggests that the internal structure of the molecules must be taken  into consideration to better understand plasma. Another well-known feature of polyatomic gases is the complicated relation of the specific heat capacity  with the temperature compared to that of monatomic gases. At  relatively low temperatures, a polyatomic gas behaves like a monatomic gas, but as the temperature rises, the rotational modes and vibrational modes are excited and the specific heat capacity increases accordingly, resulting in a step-wise temperature-capacity relation.

These considerations indicate the need for a polyatomic version of the kinetic model for dilute polyatomic plasma.
However, unlike  Boltzmann type or BGK type polyatomic kinetic equations, whose first models  can be traced back to 1894 \cite{Bryan} (See \cite{CC}), Vlasov type equations modelling polyatomic plasma have never been suggested in the literature, to the best of our knowledge. 

In the current work, as an initial step toward modelling   such a kinetic model for polyatomic plasma,
we introduce a variable extended Vlasov-Poisson equation describing the evolution of the statistical distribution of diatomic plasma in a one-dimensional configuration. This  oversimplifies the situation, but can be justified when the plasma are aligned in a strong magnetic field. We note that Degond and Liu have studied the variable extended Vlasov equation to describe  the motion of dumbbell-like polymers \cite{D-L}, and they also considered the hydrodynamic limit problem for the polymer model. We also refer to \cite{B-C-H} for statistical mechanics approaches.

\subsection{Vlasov-Poisson system for diatomic molecules with oscillatory bonds}
We now derive our model  describing the dynamics of plasma consisting of diatomic molecules, where oscillatory bonds are formed between the atoms in the molecules.
Let $x$ and $v$ denote the center of mass of the linear molecule and the translational velocity of the molecule, respectively. We also introduce $\omega$  to denote
the relative distance of the atom in the molecule from the center of mass of the molecule, and $\eta$ to denote the relative rate of change of the distance of the atom from the center of mass.
In view of this, we introduce the distribution function defined on the phase space of $(x,v,\omega,\eta)\in \times\mathbb{R}\times \mathbb{R}\times (0,\epsilon)\times \mathbb{R}$  at time $t\geq0$.
To present our model, we first introduce the local density:
\begin{align*}
\rho(t,x)=\int_{\bbr \times (0,\epsilon)\times \bbr}2f(t,x,v,\omega,\eta)dvd\omega d\eta,
\end{align*}
and the strength of fields that each atom in the molecule at the position $x\in\mathbb{R}$ experiences is given by:
\[
F(t,x)=-\frac{1}{2}\left\{\int_{-\infty}^x\rho(t,y)dy-\int^{\infty}_x\rho(t,y)dy\right\}.
\]

We then postulate that the dynamics of the molecules in the plasma is determined by the following factors.
\begin{enumerate}
\item The translational motion of the center of mass is caused the net effect of fields on the atoms in the molecules;
\begin{align*}
\frac{dx}{dt}=v,\qquad\frac{dv}{dt}=F(t,x-\omega)+F(t,x+\omega).
\end{align*}
\item The change of displacement of each atom with respect to the center of mass is explained by the combined effect of stretching due the external force differences and oscillatory bonding forces;
\begin{align*}
\frac{d\omega}{dt}=\eta,\qquad\frac{d\eta}{dt}=\left\{F(t,x-\omega)-F(t,x+\omega)\right\}+F^h(\omega).
\end{align*}
\end{enumerate}
The situation we want to model by the oscillatory bonding forces is that   atoms attract each other when they are far apart so that the atoms can be combined to form a molecule, but repel when they come too close together to keep it from collapsing. For simplicity, we assume throughout this paper that the oscillatory force $F^h(\omega): (0,\epsilon)\rightarrow\mathbb{R}$ satisfies:
\begin{enumerate}
\item[(H1)] $F^h(\omega)$ is monotone decreasing,\\ \phantom{a}such that $F^h(\omega)\to \infty $ as $\omega\to 0+$ and  $F^h(\omega)\to -\infty $ as $\omega\to \epsilon-$.
\item[(H2)] $F^h(\epsilon/2)=0$.
\item[(H3)] Symmetry of $F^h(\omega)$:  $F^h(\omega)$ is an odd function with respect to $\omega=\epsilon/2$.
\item[(H4)] Convexity  of $F^h(\omega)$: $F^h(\omega)$ is convex on $(0,\epsilon/2)$ and concave on $(\epsilon/2,\epsilon)$.
\end{enumerate}
For technical reasons, we also assume that $F^h(\omega)$ is a $C^2$-function.  
One such example is given by
\[
F^h(\omega)=-\tan \frac{\pi}{\epsilon}\Big(\omega-\frac{\epsilon}{2}\Big).
\]

The corresponding   one-dimensional diatomic Vlasov-Poisson system with oscillatory molecular bonds reads
\begin{align} \label{pVP}
(pVP) \left\{\begin{aligned}
 &\partial_t f + v  \partial_x f + F^+
\partial_v f+\eta\partial_\omega f+ (F^-+F^{h}) \partial_\eta f  =0,
 \\&F^{\pm}(t,x,\omega)=\partial_x\phi(t,x+ \omega)\pm\partial_x\phi(t,x- \omega),
 \\&F^h=F^h(\omega),
 \\&-\partial^2_x \varphi(t,x)  = \rho,  \quad \rho(t,x) = \int_{\bbr}\int_{(0,\epsilon)}\int_{\bbr}2f  dvd\omega d\eta,\\
&f(0,x,v,\omega,\eta) = f_0(x,v,\omega,\eta) \ge 0.
\end{aligned}\right.
\end{align}
Here, $f=f(t,x,v,\omega,\eta)$ denotes the particle distribution function at time $t\geq0$ on a phase point $(x,v,\omega,\eta)$ in  $(x,v)\in  \bbr\times \bbr$ and $(\omega,\eta)\in (0,\epsilon)\times \bbr $. The force fields that cause  translational and vibrational motions are given by
\begin{align}
\begin{aligned}\nonumber
\rho(t,x)&=\int_{\bbr \times (0,\epsilon)\times \bbr}2f(t,z)dvd\omega d\eta,\\
F(t,x)&=-\frac{1}{2}\left\{\int_{-\infty}^x\rho(t,y)dy-\int^{\infty}_x\rho(t,y)dy\right\},\\
F^{\pm}(t,x,\omega)&=F(t,x+ \omega)\pm F(t,x- \omega).
\end{aligned}
\end{align}


\subsection{Main theorem} We present our main result on existence of the global in time solution for (\ref{pVP}).
We first need to set up some notational conventions.\newline
\phantom{a}
{\bf Notation:} \begin{itemize}
\item We denote the whole phase space by
\[\mathcal{D}=\bbr\times \bbr\times(0,\epsilon)\times \bbr. \]

\item Let f and $\rho$ be nonnegative real valued functions defined on the whole phase space $(x,v, \omega,\eta) \in \mathcal{D}$ and  the one-dimensional real line $\bbr$, respectively. Then, we set
\begin{eqnarray*}
\|f(t)\|_{L^p_{z}}&=&\bigg(\int_\bbr\int_\bbr\int_{(0,\epsilon)}\int_\bbr |f(t,x,v,\omega,\eta)|^pd\eta d\omega dvdx\bigg)^{1/p}, \\
\|\rho(t)\|_{L^p_{x}}&=&\bigg(\int_\bbr|\rho(t,x)|^pdx\bigg)^{1/p},
\end{eqnarray*}
for $p\geq 1$.
\item For measurable functions $f(t,x,v,\omega,\eta)$, $F(t,x)$ and $F^\pm(t,x,\omega)$, we define
\begin{eqnarray*}
\|f(t)\|_{L^\infty_{z}}&=&\supess_{(x,v,\omega,\eta)\in \mathcal{D}}|f(t,x,v,\omega,\eta)|,\\
\|F(t)\|_{L^\infty_{x}}&=&\supess_{x\in \bbr}|F(t,x)|, \\
\|F^\pm(t)\|_{L^\infty_{x,\omega}}
&=&\supess_{(x,\omega)\in\bbr\times\bbr}|F^\pm(t,x,\omega)|. \end{eqnarray*}
\end{itemize}
We are ready to state our main theorem:
\begin{theorem}\label{main result}Let $\mathcal{D}=\bbr\times\bbr\times(0,\epsilon)\times\bbr$ and $\mathring{f}\in C^1_c(\mathcal{D})$  be a nonnegative function. Then there exists a unique classical solution $f(t,x,v,\omega,\eta)\in C^1((0,\infty)\times\mathcal{D})$ to \eqref{pVP}. Moreover, $f(t)$ has  compact support for all $t\in (t,\infty)$ and $f(t)\geq 0$, and
\[\|f(t)\|_{L^p_{z}}=\|\mathring{f}\|_{L^p_{z}},\quad t\geq 0\]
  for $p\geq 1$ or $p=\infty$.
\end{theorem}

For the local existence and  uniqueness, we use a standard iteration argument with compactly supported initial data. For the extension of this local solution
to the global one, the main concern is whether a particle trajectory reaches a singular point of the Hooke potential at a finite time $t>0$. Therefore, we need to carefully estimate  the distance  between two atoms in a diatomic molecule.  Since the corresponding particle trajectory is  non-autonomous and its motion is oscillatory, it is difficult to obtain such estimates.  The main idea for overcoming this difficulty is to divide the time region into a chaotic region and an almost autonomous region. More precisely, we compare the strength of the stretching force exerted on the atom from the field $F^-$ with the strength of the stretching force that comes from the oscillatory potential $F^h$, and then divide the time interval into the region in which the strengths of $F^-$ and $F^h$ are comparable (i.e., the chaotic region) and the region where the autonomous oscillatory forces dominates (i.e., the almost autonomous region).
In the almost autonomous region,  sharp estimates on the relative velocities of atoms are available through a descriptive analysis of the oscillatory motions of the atoms due
to the almost autonomous nature of the system. Although this is not the case for the chaotic region, we only need a rough estimate in this region.
By estimating the time duration for each case, we obtain the global existence result without any smallness assumption.\newline

Before we close this introduction, we briefly mention a historical remark on the Cauchy problem of the monatomic Vlasov-Poisson system.  Ukai and Okabe \cite{U-O} studied the global existence and uniqueness for the two-dimensional case with a small initial data assumption. Horst-Hunze-Neunzert \cite{H-H-N}  proved the existence of the weak solution for the three-dimensional system. Batt \cite{Batt} and Bardos-Degond \cite{B-D} obtained the existence of symmetric classical solutions and  small initial data solutions, respectively. After their works, Lions-Perthame \cite{Lions-Perthame} and Pfaffelmoser \cite{Pfaffelmoser} independently obtained the   global-in-time  solution for three-dimensional large data. We also refer to \cite{Guo, Hwang, H-V1, H-V2} for the initial boundary value problems. To the best of our knowledge, we are not aware of any existence results on polyatomic Vlasov equations.\newline

This paper is organized as follows. In Section \ref{sec2}, we consider a solution of the linear diatomic Vlasov equation via a standard particle trajectory argument. In Section \ref{sec3}, we provide the local existence,  uniqueness and continuation principle of the solution to the diatomic Vlasov-Poisson system.  In Section \ref{sec4}, we present estimates for the oscillatory variables $\Omega$ and $H$. Finally, Section \ref{sec5} is devoted to the proof of our global existence result.

\section{Particle trajectory and linear Vlasov equation}\label{sec2}
\setcounter{equation}{0}
We consider $F^+(t, \cdot,\cdot)$ and $F^-(t, \cdot,\cdot)$ as  given vector fields with Lipschitz continuity. We assume that $F^h$ satisfies  conditions (H1)-(H4) in Section \ref{intro}. Let $f=f(t,x,v,\omega,\eta)$ be the solution to  the following linear diatomic Vlasov (ldV) equation.
\begin{align} \label{linearpVP}
(ldV) \left\{\begin{aligned}
 &\partial_t f + v  \partial_x f + F^+
\partial_v f+\eta\partial_\omega f+ (F^-+F^{h}) \partial_\eta f  =0,
 \\&F^h=F^h(\omega),
 \\&f(0,x,v,\omega,\eta) = f_0(x,v,\omega,\eta) \ge 0.
\end{aligned}\right.
\end{align}

 In this part, we study  particle trajectories generated by $(v,F^+(x,\omega),\eta,F^-(x,\omega) +F^h(\omega) )$ on $(x,v)\in\bbr\times\bbr$ and $(\omega,\eta)\in (0, \epsilon)\times \bbr$. From the particle trajectories, we derive the solution formula for the linear diatomic Vlasov equation and obtain its measure preserving property.

We fix $t\in \bbr_+$ and  let $z=[x,v,\omega,\eta]$. We define the particle trajectory (PT) by
\[[X(s;t,z),V(s;t,z)]=[X(s;t,x,v,\omega,\eta), V(s;t,x,v,\omega,\eta)]\] and \[[\Omega(s;t,z), H(s;t,z)]=[\Omega(s;t,x,v,\omega,\eta), H(s;t,x,v,\omega,\eta)]\] passing through $(x,v)$ and $(\omega,\eta)$  at time $t$, respectively,  satisfying the following  the ordinary differential equation (ODE) system:

\begin{align} \label{linearpt}
(PT) \left\{\begin{aligned}\frac{d}{ds}X(s;t,z) &= V(s;t,z),\\
 \frac{d}{ds}V(s;t,z)& = F^+ (s,X(s;t,z),\Omega(s;t,z)),\\
\frac{d}{ds}\Omega(s;t,z) &= H(s;t,z),\\
 \frac{d}{ds}H(s;t,z) &=
F^-(s,X(s;t,z),\Omega(s;t,z)) +F^h(\Omega(s;t,z)),
\end{aligned}\right.
\end{align}
subject to the initial data
\[  X(t;t,z) = x \in \bbr\quad \mbox{and} \quad  V(t;t,z) = v \in \bbr, \]
and
\[  \Omega(t;t,z) = \omega\in (0,\epsilon) \quad \mbox{and} \quad  H(t;t,z) = \eta \in \bbr, \]
where $F^+(s, \cdot,\cdot)$ and $F^-(s, \cdot,\cdot)$ are  given vector fields and $F^h(\cdot)$ is a given Hooke potential function.

Then, for the solution $Z(s;t,z)=[X(s;t,z),V(s;t,z),\Omega(s;t,z),H(s;t,z)]$ to the ODE system \eqref{linearpt}, the following relation holds:
\[\frac{d}{ds}f(s,Z(s;t,z))=\partial_t f +v\cdot \partial_xf +F^+\cdot\partial_vf +\eta\cdot \partial_\omega f +(F^-+F^{h})\cdot\partial_\eta f=0. \]
Thus, similar to the classical monatomic Vlasov equation, we have
\[\mathring{f}(Z(0;t,z))=f(0,Z(0;t,z))=f(t,Z(t;t,z))=f(t,x,v,\omega,\eta).\]

For $z=[x,v,\omega,\eta]$, we can calculate  $\det \partial Z/\partial z(s;t,z)$ as
\[\det \frac{\partial Z}{\partial z}(t;t,z)=1\]
 and
\[\frac{d}{ds}\det \frac{\partial Z}{\partial z}(s;t,z)=\partial_z\cdot G(s,Z(s;t,z)) \det \frac{\partial Z}{\partial z}.\]
Note that
\[G(x,v,\omega,\eta)=[v, ~F^+(s,x,\omega), ~\eta, ~F^-(s,x,\omega)+F^h(\omega)].\]

Since the vector field $G$ is divergence-free, which is the same as in the classical monatomic Vlasov case, the measure preserving property also holds.

From this measure preserving property, we have
\begin{eqnarray*}
\int f^p(t,x,v,\omega,\eta)dxdvd\omega d\eta &=&\int f^p(0,Z(0;t,z))dxdvd\omega d\eta\\&=&\int f^p(0,Z(0;t,z))\det \frac{\partial Z}{\partial z}(0;t,z) dxdvd\omega d\eta\\
&=&\int f^p(0,x,v,\omega,\eta)dxdvd\omega d\eta.
\end{eqnarray*}

We summarize the results in this section as follows.
\begin{proposition}\label{prop2.1}
Let $f$ be the solution to \eqref{linearpVP} for given vector fields $F^\pm$, $F^h$, and  \[Z(s;t,z)=[X(s;t,z),V(s;t,z),\Omega(s;t,z),H(s;t,z)]\] be   the solution  to the ODE system \eqref{linearpt}. We assume that $\Omega(s;t,z)\in (0,\epsilon)$ on $0\leq s,t\leq T$.

Then the solution $f$ to \eqref{linearpVP} satisfies
\[f(t,x,v,\omega,\eta)=\mathring{f}(Z(0;t,z)),\]
and
\begin{eqnarray*}
\|f(t)\|_{L^p_{z}} &=&\|\mathring{f}\|_{L^p_{z}},\quad \mbox{for}\quad 0<t<T,
\end{eqnarray*}
where  $p\geq 1$ or  $p=\infty$.

\end{proposition}
\section{Local existence}\label{sec3}
\setcounter{equation}{0}
In this part, we consider the local existence of the solution to \eqref{pVP}. We will follow the standard argument in \cite{Rein}.  Let $\mathring{f}(z)=f(0,z)\in C^1_c(\bbr\times\bbr \times (0,\epsilon)\times \bbr)$ with $f(0,z)\geq0$. Thus, there exist $\mathring{P^x}$, $\mathring{P^v}$, $\mathring{P^{\omega\pm}}$     and  $\mathring{P^\eta}$  such that
\begin{eqnarray}\label{eq 4.0}
f(0,z)=0,\quad |x|\geq \mathring{P^x} ~\mbox{or}~ |v|\geq \mathring{P^v}~\mbox{or} ~ 0<\omega \leq \mathring{P^{\omega-}}   ~\mbox{or} ~ \epsilon> \omega \geq \mathring{P^{\omega+}}   ~\mbox{or} ~|\eta|\geq \mathring{P^\eta}.
\end{eqnarray}

We will use a well-known iterative scheme to obtain the local existence result.  First, we define $f_0$ as follows.
\[f_0(t,z)=f(0,z)=\mathring{f}(z), \quad t\geq0.\]

For a  given $f_n\in C^1_c([0,\infty)\times \bbr\times\bbr \times (0,\epsilon)\times \bbr)$, we define $f_{n+1}$ inductively as follows.

\noindent Let $\rho_n$ and $F^{\pm}_n$ be the mass density and force fields for the given phase distribution $f_n$ satisfying

\begin{align}
\begin{aligned}\label{inductionD}
\rho_n(t,x)&=\int_{\bbr \times (0,\epsilon)\times \bbr}2f_n(t,z)dvd\omega d\eta,\\
F_n(t,x)&=\frac{1}{2}\Big(\int^{\infty}_x\rho_n(t,y)dy-\int_{-\infty}^x\rho_n(t,y)dy\Big),\\
F^{\pm}_n(t,x,\omega)&=F_n(t,x+ \omega)\pm F_n(t,x- \omega)
,
\end{aligned}
\end{align}
and
$F^h=F^h(\omega)$ is a given function satisfying conditions (H1)-(H4).

Then we define $f_{n+1}$ as the solution to the following linear diatomic Vlasov equation:
\begin{align}
 \begin{aligned}\nonumber
 &\partial_t f_{n+1} + v  \partial_x f_{n+1}+ F^+_n
\partial_v f_{n+1}+\eta\partial_\omega f_{n+1}+ (F^-_n+F^h) \partial_\eta f_{n+1} =0,
\end{aligned}\end{align}
subject to the initial data $f(0,x,v,\omega,\eta) = \mathring{f}(x,v,\omega,\eta) \ge 0$.
Then by Proposition \ref{prop2.1} in  Section \ref{sec2}, we have the following solution formula:
\[f_{n+1}(t,z)=\mathring{f}(Z_n(0;t,z)),\quad t\geq0, \quad z=(x,v,\omega,\eta),\]
where $Z_n(s;t,z)$ is the particle trajectory defined by
\begin{align}
(PTn) \left\{ \begin{aligned}\label{ptn}\frac{d}{ds}X_n(s;t,z) &= V_n(s;t,z),\\
\frac{d}{ds}V_n(s;t,z)& = F^+_n (s,X_n(s;t,z),\Omega_n(s;t,z)),\\
\frac{d}{ds}\Omega_n(s;t,z) &= H_n(s;t,z),\\
 \frac{d}{ds}H_n(s;t,z) &= F^-_n(s,X_n(s;t,z),\Omega_n(s;t,z)) +F^h(\Omega_n(s;t,z)),
\end{aligned}\right.\end{align}subject to the initial data
\[  X_n(t;t) = x \in \bbr\quad \mbox{and} \quad  V_n(t;t) = v \in \bbr, \]
and
\[  \Omega_n(t;t) = \omega\in (0,\epsilon) \quad \mbox{and} \quad  H_n(t;t) = \eta \in \bbr. \]

%
%
%
%
%


Similar to $f_n$, we inductively define the  support boundaries corresponding to $f_n$  as
\begin{eqnarray*}
P^x_0=\mathring{P^x}, \quad P^x_n(t):=\sup\{|X_{n-1}(s,0,z)|:z\in supp ~\mathring{f},~ 0\leq s\leq t\},
\\ P^v_0=\mathring{P^v}, \quad P^v_n(t):=\sup\{|V_{n-1}(s,0,z)|:z\in supp ~\mathring{f},~ 0\leq s\leq t\},\\
P^\eta_0=\mathring{P^\eta}, \quad P^\eta_n(t):=\sup\{|H_{n-1}(s,0,z)|:z\in supp ~\mathring{f},~ 0\leq s\leq t\},\\
\end{eqnarray*}
and $P^{\omega\pm}_0=\mathring{P^{\omega\pm}},$
\begin{eqnarray*}
 && P^{\omega-}_n(t):=\inf\{\Omega_{n-1}(s,0,z):z\in supp ~\mathring{f},~ 0\leq s\leq t\},\\
&&P^{\omega+}_n(t):=\sup\{\Omega_{n-1}(s,0,z):z\in supp ~\mathring{f},~ 0\leq s\leq t\}.
\end{eqnarray*}

\subsection{Boundedness of $\Omega$ for small time $0\leq s\leq t_0$:}

We assume that
\[\Omega_n(0;0,z)=\omega \in (0+\epsilon_0, \epsilon-\epsilon_0),\quad x\in (-R,R), \quad v\in (-R,R),\quad  \mbox{and}~ \eta\in (-R,R),\]
for some $0<2\epsilon_0<\epsilon$ and $R>0$.
Then we consider the following linear function depending on $\epsilon_0$:
\[G(x)=\frac{2x-\epsilon}{\epsilon_0-\epsilon}F^h\Big(\frac{\epsilon_0}{2}\Big).\]
Then we have
\[ |F^h(x)|\leq|G(x)|, \quad \mbox{for}\quad \frac{\epsilon_0}{2}\leq x\leq\epsilon-\frac{\epsilon_0}{2}.   \]

\begin{lemma}\label{lemma 4.1}Let $Z_n(s;t,z)=[X_n(s;t,z), V_n(s;t,z),\Omega_n(s;t,z), H_n(s;t,z)]$ be the solution to \eqref{ptn} and let $0<2\epsilon_0<\epsilon$, $n\in\bbn$ and
\[\Omega_n(0;0,z)=\omega,\quad H_n(0;0,z)=\eta.\]
We a priori assume that
\begin{eqnarray}\label{eq 4.2}
\frac{\epsilon_0}{2}\leq\Omega_n(s;0,z)\leq \epsilon-\frac{\epsilon_0}{2}
\end{eqnarray}
and there is a positive constant $C_->0$ such that
\begin{eqnarray}\label{eq 4.3}
|F^-(s,X_n(s;0,z),\Omega_n(s;0,z))|\leq C_-.
\end{eqnarray}

Then we have
\begin{eqnarray*}
|\Omega_n(t;0,z)|+ |H_n(t;0,z)|&\leq &
e^{C_1t}(|\omega|+|\eta|)+C_2\frac{e^{C_1t}-1}{C_1}
,\end{eqnarray*}
where $C_1$ and $C_2$ are constants depending on $\epsilon$, $\epsilon_0$, $C_-$ and $F^h$.

\end{lemma}
\begin{proof}

From elementary calculation and the definition of particle trajectory,  the following differential inequalities hold:
\begin{align}
\begin{aligned}\label{eq 4.1}
\frac{d}{ds}|\Omega_n(s;0,z)| &\leq  |H_n(s;0,z)|,\\
\quad \frac{d}{ds}|H_n(s;0,z)| &\leq
|F^-(s,X_n(s;0,z),\Omega_n(s;0,z))| +|F^h(\Omega_n(s;0,z))|.
\end{aligned}
\end{align}

We now add  both inequalities in \eqref{eq 4.1} to obtain
\begin{eqnarray*}
\frac{d}{ds}|\Omega_n(s;0,z)| +\frac{d}{ds}|H_n(s;0,z)|&\leq & |H_n(s;0,z)|
+|F^-(s,X_n(s;0,z),\Omega_n(s;0,z))| +|F^h(\Omega_n(s;0,z))|.
\end{eqnarray*}
Note that by the assumption in \eqref{eq 4.2},
\[ |F^h(\Omega_n(s;0,z))|\leq|G(\Omega_n(s;0,z))|.   \]

Since we assume that $|F^-(s,X_n(s;0,z),\Omega_n(s;0,z))|\leq C_-$ in \eqref{eq 4.3},
we have
\begin{eqnarray*}
\frac{d}{ds}|\Omega_n(s;0,z)| +\frac{d}{ds}|H_n(s;0,z)|&\leq & |H_n(s;0,z)|
+C_- +|G(\Omega_n(s;0,z))|\\
&=&|H_n(s;0,z)|
+C_-
+\bigg|F^h\Big(\frac{\epsilon_0}{2}\Big)\frac{2\Omega_n(s;0,z)-\epsilon}{\epsilon_0-\epsilon}\bigg|
\\
&\leq&
|H_n(s;0,z)|
+C_-
+2\Big|F^h\Big(\frac{\epsilon_0}{2}\Big)\Big|\bigg|\frac{\Omega_n(s;0,z)}{\epsilon_0-\epsilon}\bigg|+\Big|F^h\Big(\frac{\epsilon_0}{2}\Big)\Big|\bigg|\frac{\epsilon}{\epsilon_0-\epsilon}\bigg|.
\end{eqnarray*}

For simplicity, we denote
\begin{align}
\begin{aligned}\label{eq 4.33}
&C_1=C_1(\epsilon,\epsilon_0,F^h)=\max\bigg\{1,2\bigg|\frac{F^h(\frac{\epsilon_0}{2})}{\epsilon_0-\epsilon}\bigg|\bigg\},\\
&C_2=C_2(\epsilon,\epsilon_0,F^h)=C_-+\frac{\epsilon}{2}C_1.
\end{aligned}
\end{align}
Using the notation, it follows that
\begin{eqnarray}\label{eq 4.15}
\frac{d}{ds}(|\Omega_n(s;0,z)|+ |H_n(s;0,z)|)&\leq &C_1(|\Omega_n(s;0,z)| +|H_n(s;0,z)|)+C_2,\end{eqnarray}
where $C_1$ and $C_2$ are the positive constants in \eqref{eq 4.33}.

We multiply both sides of \eqref{eq 4.15} by $e^{-C_1s}$ and integrate them over $[0,s]$ to obtain
\begin{eqnarray*}
e^{-C_1s}(|\Omega_n(s;0,z)|+ |H_n(s;0,z)|)-(|\Omega_n(0;0,z)|+ |H_n(0;0,z)|)&\leq & C_2\int_0^se^{-C_1\tau}d\tau.\end{eqnarray*}

This implies that
\begin{eqnarray*}
|\Omega_n(s;0,z)|+ |H_n(s;0,z)|&\leq &
e^{C_1 s}(|\omega|+|\eta|)+C_2e^{C_1 s}\int_0^se^{-C_1\tau}ds\\
&=&
e^{C_1s}(|\omega|+|\eta|)+C_2\frac{e^{C_1s}-1}{C_1}
.\end{eqnarray*}
\end{proof}

\begin{lemma}\label{lemma 4.2}Let $Z_n(s;t,z)=[X_n(s;t,z), V_n(s;t,z),\Omega_n(s;t,z), H_n(s;t,z)]$ be the solution to \eqref{ptn} and let $0<2\epsilon_0<\epsilon$, $n\in\bbn$ and
\[\Omega_n(0;0,z)=\omega,\quad H_n(0;0,z)=\eta.\]
With the same assumptions of Lemma \ref{lemma 4.1}, we have
\begin{eqnarray*}
\big|\Omega_n(s;0,z)-\Omega_n(0;0,z)\big|\leq se^{C_1s}(|\omega|+|\eta|)+C_2s \frac{e^{C_1s}-1}{C_1},
\end{eqnarray*}
where $C_1$ and $C_2$ are constants depending on $\epsilon$, $\epsilon_0$, uniform upper bound of $|F^-|$ and $F^h$.

\end{lemma}
\begin{proof}
By Lemma \ref{lemma 4.1},
\begin{eqnarray*}
 |H_n(s;0,z)|&\leq &
e^{C_1s}(|\omega|+|\eta|)+C_2\frac{e^{C_1s}-1}{C_1}
.\end{eqnarray*}

Note that
\[|\Omega_n(s;0,z)-\Omega_n(0;0,z)| = \Big|\int^s_0 H_n(\tau;0,z)d\tau\Big|.\]

Thus,
\begin{eqnarray*}
|\Omega_n(s;0,z)-\Omega_n(0;0,z)|\leq se^{C_1s}(|\omega|+|\eta|)+C_2s \frac{e^{C_1s}-1}{C_1}.
\end{eqnarray*}

\end{proof}

\begin{proposition}\label{prop 4.3}
Let $Z_n(s;t,z)=[X_n(s;t,z), V_n(s;t,z),\Omega_n(s;t,z), H_n(s;t,z)]$ be the solution to \eqref{ptn}.
We assume that for some $0<2\epsilon_0<\epsilon$ and $R>0$,
\[\Omega_n(0;0,z)=\omega \in [0+\epsilon_0, \epsilon-\epsilon_0]\quad  \mbox{and}\quad H_n(0;0,z)=\eta\in [-R,R].\]

Then there exists a small number $t_0>0$ such that
\[ \frac{\epsilon_0}{2}\leq \Omega_n(s;0,z)\leq  \epsilon-\frac{\epsilon_0}{2},\]
for all $0\leq s\leq t_0$.
Here $t_0>0$ depends on $R$, $\epsilon$, $\epsilon_0$, $F^-$ and $F^h$.

\end{proposition}
\begin{proof}Let $C_1$ and $C_2$ be constant defined in Lemma \ref{lemma 4.1}. Note that  $C_1$ and $C_2$ depend on $\epsilon$, $\epsilon_0$, $F^-$ and $F^h$. Assume that
\[ \epsilon_0\leq \omega\leq \epsilon-\epsilon_0,\quad  |\eta|\leq R. \]

From elementary calculation, it follows that
\[\lim_{t\to 0}\bigg[te^{C_1t}(\epsilon+R)+C_2t \frac{e^{C_1t}-1}{C_1}\bigg]=0,\]
which implies that there is a number $t_0>0$ such that for $0\leq s\leq t_0$,
\begin{eqnarray}\label{eq 4.5}\bigg[se^{C_1s}(\epsilon+R)+C_2s \frac{e^{C_1s}-1}{C_1}\bigg]< \frac{\epsilon_0}{4}.\end{eqnarray}

Since the initial value $\Omega_n(0;0,z)$ satisfies
\begin{eqnarray*}
 \epsilon_0\leq \Omega_n(0;0,z)=\omega\leq \epsilon-\epsilon_0,
\end{eqnarray*}
  for each $\omega$ with fixed $n\in \bbn$, there is $t_\omega>0$ such that
\[ \frac{\epsilon_0}{2}\leq \Omega_n(s;0,z)\leq  \epsilon-\frac{\epsilon_0}{2},\quad 0\leq s\leq t_\omega.\]

Let $(0,T_\omega)$ be the maximal interval  satisfying the above for given $\omega=\Omega_n(0;0,z)\in [\epsilon_0,\epsilon-\epsilon_0]$, i.e.,
\[T_\omega=\sup\{t_\omega: \frac{\epsilon_0}{2}\leq \Omega_n(s;0,z)\leq  \epsilon-\frac{\epsilon_0}{2},~0\leq s\leq t_\omega \}.\]

Claim: For any $\omega\in[\epsilon_0,\epsilon-\epsilon_0]$, we have
\[\inf_{\omega\in [\epsilon_0,\epsilon-\epsilon_0] }T_\omega\geq t_0>0,\]
where $t_0$ is defined as above.
By way of contradiction, suppose not. Then there is an initial value   $\omega_0\in [\epsilon_0,\epsilon-\epsilon_0]$ and $\eta_0 \in [-R,R]$ such that
\[T_{\omega_0}< t_0.\]

On $0\leq s < T_{\omega_0}$, by the definition of $T_\omega$,
\[ \frac{\epsilon_0}{2}\leq \Omega_n(s;0,z)\leq  \epsilon-\frac{\epsilon_0}{2}.\]

Lemma \ref{lemma 4.2} implies that
\begin{eqnarray*}
|\Omega_n(s;0,z)-\Omega_n(0;0,z)|\leq se^{C_1s}(|\omega_0|+|\eta_0|)+C_2s \frac{e^{C_1s}-1}{C_1},\quad \mbox{on $0\leq s < T_{\omega_0}$.
}
\end{eqnarray*}

We note that
\[f(s)=se^{C_1s}(\epsilon+R)+C_2s \frac{e^{C_1s}-1}{C_1}\]
is an increasing function.

Thus, on $0\leq s < T_{\omega_0}$,
\begin{eqnarray*}
|\Omega_n(s;0,z)-\Omega_n(0;0,z)|\leq se^{C_1s}(|\omega_0|+|\eta_0|)+C_2s \frac{e^{C_1s}-1}{C_1}=f(s)\leq f(T_{\omega_0}).
\end{eqnarray*}
By \eqref{eq 4.5}, we have
\[f(T_{\omega_0})\leq f(t_0)\leq \frac{\epsilon_0}{4},\]
which implies that
\begin{eqnarray*}
|\Omega_n(s;0,z)-\Omega_n(0;0,z)|\leq \frac{\epsilon_0}{4}.
\end{eqnarray*}

Therefore,  on $0\leq s \leq  T_{\omega_0}$,
\begin{eqnarray*}
\frac{3\epsilon_0}{4}=-\frac{\epsilon_0}{4}+\epsilon_0\leq -\frac{\epsilon_0}{4}+\omega_0 \leq \Omega_n(s;0,z)\leq \frac{\epsilon_0}{4}+\omega_0\leq\frac{\epsilon_0}{4}+ \epsilon-\epsilon_0= \epsilon-\frac{3\epsilon_0}{4}.
\end{eqnarray*}
Thus, we have
\[\lim_{s\to T_{\omega_0}-} \Omega_n(s;0,z)\in \Big[\frac{3\epsilon_0}{4},\epsilon-\frac{3\epsilon_0}{4}\Big].\]

This  contradicts the definition of $T_{\omega_0}$. Therefore, we have proven that for  any $\omega\in[\epsilon_0,\epsilon-\epsilon_0]$,
\[T_\omega\geq t_0.\]
Thus, we obtain the desired result.

\end{proof}

\subsection{Convergence:}
In this subsection, we prove the sequence $f_n$ converges uniformly on $(0,t_0)$, where $t_0>0$ is a small number defined as in Proposition \ref{prop 4.3}.

\begin{lemma}\label{lemma 4.4}For  given initial data $\mathring{f}$, let $F_n$ be the function defined in \eqref{inductionD}. Then
\begin{eqnarray*}\|F^{\pm}_n(t)\|_{L^\infty_{x,\omega}}\leq 2\|\mathring{f}\|_{L^1_{z}}.\end{eqnarray*}
\end{lemma}
\begin{proof}
By the measure preserving property in the previous section, the following holds
\[\|\rho_{n}\|_{L^1_{x}} =\|\mathring{f}\|_{L^1_{z}}.\]

Therefore, we have
\begin{eqnarray*}
|F_n(t,x)|&=&\frac{\big|\int_{-\infty}^x\rho_n(t,y)dy-\int^{\infty}_x\rho_n(t,y)dy\big|}{2}\\
&\leq&\frac{\int_{-\infty}^x\rho_n(t,y)dy+\int^{\infty}_x\rho_n(t,y)dy}{2}\\
&=&\frac{1}{2}\|\rho_{n}\|_{L^1_{x}}\\
&=&  \|\mathring{f}\|_{L^1_{z}}.
\end{eqnarray*}
Thus,
\begin{eqnarray*}\|F^{\pm}_n(t)\|_{L^\infty_{x,\omega}}\leq 2\|\mathring{f}\|_{L^1_{z}}.\end{eqnarray*}

\end{proof}
\begin{remark}
Since \[f_{n}(t,z)=\mathring{f}(Z_{n-1}(0;t,z)),\quad t\geq0, \quad z=(x,v,\omega,\eta)\] and
we assume  \eqref{eq 4.0}, we have
\[f_{n}(t,z)=0\quad \mbox{for}\quad  \omega\leq P^{\omega-}_{n-1}(t)\quad\mbox{or}\quad  \omega\geq P^{\omega+}_{n-1}(t).  \]
\end{remark}

\begin{lemma}\label{pFb}For a given $0<2\epsilon_0<\epsilon$, we assume that
\[\Omega_n(0;0,z)=\omega \in [0+\epsilon_0, \epsilon-\epsilon_0]\quad  \mbox{and}\quad H_n(0;0,z)=\eta\in [-R,R].\]

Let $[X_n,V_n,\Omega_n,H_n]$ be the particle trajectory that satisfies \eqref{ptn} with $f_n,\rho_n,F_n$ in \eqref{inductionD}. Then
 \[\bigg\|\frac{\partial F_n(t)}{\partial x}\bigg\|_{L^\infty_{x}} \leq C,\quad 0\leq t \leq t_0,\]
where $t_0$ is a small number defined as in Proposition \ref{prop 4.3}. Moreover,
\begin{eqnarray*}
\Big\|\frac{\partial F_n^+(t)}{\partial x}\Big\|_{L^\infty_{x,\omega}}+ \Big\|\frac{\partial F_n^+(t)}{\partial \omega}\Big\|_{L^\infty_{x,\omega}}+\Big\|\frac{\partial F_n^-(t)}{\partial x}\Big\|_{L^\infty_{x,\omega}}+ \Big\|\frac{\partial F_n^-(t)}{\partial \omega}\Big\|_{L^\infty_{x,\omega}}<C=C(\mathring{f},\epsilon_0)
,~\quad 0\leq t \leq t_0,
\end{eqnarray*}
and
 on $0\leq t\leq t_0$,
\[ P^v(t)\leq R +\|\mathring{f}\|_{L^1_{z}}t_0\]
and
\[P^\eta\leq C(\mathring{f},\epsilon_0). \]

\end{lemma}
\begin{proof}
From equation \eqref{ptn} for the particle trajectory, we have
\begin{eqnarray}\label{eq 4.8}
|V_{n}(t;0,z)|\leq |v|+\int_0^t \|F^+_{n}(\tau)\|_{L^\infty_{x,\omega}} d\tau \leq |v| +\int^t_0 \|\mathring{f}\|_{L^1_{z}}d\tau\leq R +\|\mathring{f}\|_{L^1_{z}}t,
\end{eqnarray}
and by the previous part, for the given small $\epsilon_0>0$, there is $t_0>0$ such that for $0\leq t\leq t_0$,
\begin{eqnarray}\label{eq 4.9}
|H_{n}(t;0,z)|\leq C(\mathring{f},\epsilon_0)
\end{eqnarray}
and
\[\frac{\epsilon_0}{2}\leq |\Omega_{n}(t;0,z)|\leq \epsilon-\frac{\epsilon_0}{2}. \]

We claim that there is a constant $C>0$ such that
\[\bigg\|\frac{\partial F_n(t)}{\partial x}\bigg\|_{L^\infty_{x}} \leq C,\quad 0\leq t \leq t_0.\]

By the definition of $F_n$, we have
\begin{eqnarray}\label{Frho}\frac{\partial F_{n}(t)}{\partial x}=\frac{1}{2}\frac{\partial}{\partial x}\bigg(\int_{-\infty}^x\rho_n(t,y)dy-\int^{\infty}_x\rho_n(t,y)dy\bigg)=\rho_n(t,x).\end{eqnarray}

This implies
\begin{eqnarray*}
\bigg\|\frac{\partial F_{n}(t)}{\partial x}\bigg\|_{{L^\infty_{x}}}&\leq &\|\rho_{n}(t)\|_{L^\infty_{x}}\\
&=&\Big\|\int\int\int 2f_n(x,v,\omega,\eta) dvd\omega d\eta\Big\|_{L^\infty_{x}}\\
&=&2\Big\|\int \int\int\mathring{f}(Z_n(0;t,z)) dvd\omega d\eta\Big\|_{L^\infty_{x}}\\
&=&2\Big\|\int_{|\eta|\leq P^\eta_n(t)}  \int_{(0,\epsilon)}\int_{|v|\leq P^v_n(t)}  \mathring{f}(Z_n(0;t,z)) dvd\omega d\eta\Big\|_{L^\infty_{x}}\\
&\leq& 8\epsilon P_n^v(t)P_n^\eta(t)\|\mathring{f}\|_{L^\infty_{z}}.
\end{eqnarray*}

By \eqref{eq 4.8} and \eqref{eq 4.9}, for $0\leq t\leq t_0$ and $n\in \mathbb{N}$,
\[ P_n^v(t)\leq R +\|\mathring{f}\|_{L^1_{z}}t_0\]
and
\[P_n^\eta\leq C(\mathring{f},\epsilon_0). \]
Therefore,
\[\bigg\|\frac{\partial F_n(t)}{\partial x}\bigg\|_{L^\infty_{x}} \leq C= C(\mathring{f},\epsilon_0),\quad 0\leq t \leq t_0,\]
where the constant $C$ depends on the initial data $\mathring{f}$ and $\epsilon_0$.

Since we defined $F^\pm$  as
\begin{eqnarray*}
F^{\pm}_n(t,x,\omega)=F_n(t,x+ \omega)\pm F_n(t,x- \omega),
\end{eqnarray*}
we obtain that if $0\leq t\leq t_0$,
\begin{eqnarray*}
\Big\|\frac{\partial F_n^+(t)}{\partial x}\Big\|_{L^\infty_{x,\omega}}+ \Big\|\frac{\partial F_n^+(t)}{\partial \omega}\Big\|_{L^\infty_{x,\omega}}+\Big\|\frac{\partial F_n^-(t)}{\partial x}\Big\|_{L^\infty_{x,\omega}}+ \Big\|\frac{\partial F_n^-(t)}{\partial \omega}\Big\|_{L^\infty_{x,\omega}}<C= C(\mathring{f},\epsilon_0).
\end{eqnarray*}

\end{proof}
\begin{lemma}\label{pb}
Let $[X_n,V_n,\Omega_n,H_n]$ be the solution to \eqref{ptn} and $f_n,\rho_n,F_n$ be defined as in \eqref{inductionD}. For a given $0<2\epsilon_0<\epsilon$, we assume that
\[\Omega_n(0;0,z)=\omega \in [0+\epsilon_0, \epsilon-\epsilon_0]\quad  \mbox{and}\quad H_n(0;0,z)=\eta\in [-R,R].\]
Then
\begin{eqnarray*}|\partial X_n(s;t,z)|+|\partial V_n(s;t,z)|+ |\partial \Omega_n(s;t,z)|+ |\partial H_n(s;t,z)|<C= C(\mathring{f},\epsilon_0),\quad 0\leq s<t<t_0,\end{eqnarray*}
where $t_0>0$ is a small number defined as in Proposition \ref{prop 4.3} and $\partial$ can be chosen  as a differential operator in $\big\{\frac{\partial}{\partial x},\frac{\partial}{\partial v},\frac{\partial}{\partial \omega},\frac{\partial}{\partial \eta}\big\}$. Moreover,
\begin{eqnarray*}
&&\Big\|\frac{\partial^2 F^+_n(t)}{\partial x^2}\Big\|_{L^\infty_{x,\omega}}+ \Big\|\frac{\partial^2 F^+_n(t)}{\partial x\partial \omega}\Big\|_{L^\infty_{x,\omega}}+\Big\|\frac{\partial^2 F^+_n(t)}{\partial \omega^2}\Big\|_{L^\infty_{x,\omega}}
\\
&&\qquad+\Big\|\frac{\partial^2 F^-_n(t)}{\partial x^2}\Big\|_{L^\infty_{x,\omega}}+ \Big\|\frac{\partial^2 F^-_n(t)}{\partial x\partial \omega}\Big\|_{L^\infty_{x,\omega}}+\Big\|\frac{\partial^2 F^-_n(t)}{\partial \omega^2}\Big\|_{L^\infty_{x,\omega}}<C,\quad \mbox{on $0\leq s<t<t_0$.
}\end{eqnarray*}

\end{lemma}
\begin{proof}Take the time derivatives of the particle trajectories to obtain
\begin{eqnarray*}
\frac{d}{ds}\partial X_n(s;t,z) &=& \partial V_n(s;t,z),\\
\frac{d}{ds}\partial V_n(s;t,z)& =& \partial_x F^+_n (s,X_n(s;t,z),\Omega_n(s;t,z))\partial X_n(s;t,z)\\&&+\partial_\omega F^+_n (s,X_n(s;t,z),\Omega_n(s;t,z))\partial \Omega_n(s;t,z),\\
\frac{d}{ds}\partial\Omega_n(s;t,z) &=&\partial H_n(s;t,z),\\
 \frac{d}{ds}\partial H_n(s;t,z) &=& \partial_x F^-_n (s,X_n(s;t,z),\Omega_n(s;t,z))\partial X_n(s;t,z)\\&&+\partial_\omega F^-_n (s,X_n(s;t,z),\Omega_n(s;t,z))\partial \Omega_n(s;t,z)+\frac{\partial F^h}{\partial x}(\Omega_n(s;t,z))\partial \Omega_n(s;t,z).
\end{eqnarray*}
Note that by Lemma \ref{pFb} and the result in Proposition \ref{prop 4.3},
\begin{eqnarray*}
\Big\|\frac{\partial F_n^+(t)}{\partial x}\Big\|_{L^\infty_{x,\omega}}, \Big\|\frac{\partial F_n^+(t)}{\partial \omega}\Big\|_{L^\infty_{x,\omega}},\Big\|\frac{\partial F_n^-(t)}{\partial x}\Big\|_{L^\infty_{x,\omega}}, \Big\|\frac{\partial F_n^-(t)}{\partial \omega}\Big\|_{L^\infty_{x,\omega}}<C,\quad 0\leq s<t<t_0,
\end{eqnarray*}
and
\[\frac{\partial F^h}{\partial x}(\Omega_n(s;t,z))<C,\quad 0\leq s<t<t_0.\]

By Gronwell's lemma, we have
\begin{eqnarray*}|\partial X_n(s;t,z)|+|\partial V_n(s;t,z)|+ |\partial \Omega_n(s;t,z)|+ |\partial H_n(s;t,z)|<C,\quad \mbox{on $0\leq s<t<t_0$.
}\end{eqnarray*}
From this result with the argument used in the proof of Lemma \ref{pFb}, it follows that
\begin{eqnarray*}
&&\Big\|\frac{\partial^2 F^+_n(t)}{\partial x^2}\Big\|_{L^\infty_{x,\omega}}+ \Big\|\frac{\partial^2 F^+_n(t)}{\partial x\partial \omega}\Big\|_{L^\infty_{x,\omega}}+\Big\|\frac{\partial^2 F^+_n(t)}{\partial \omega^2}\Big\|_{L^\infty_{x,\omega}}\\
&&\qquad
+\Big\|\frac{\partial^2 F^-_n(t)}{\partial x^2}\Big\|_{L^\infty_{x,\omega}}+ \Big\|\frac{\partial^2 F^-_n(t)}{\partial x\partial \omega}\Big\|_{L^\infty_{x,\omega}}+\Big\|\frac{\partial^2 F^-_n(t)}{\partial \omega^2}\Big\|_{L^\infty_{x,\omega}}<C,\quad \mbox{on $0\leq s<t<t_0$.
}\end{eqnarray*}

\end{proof}

\begin{lemma}\label{lemma 4.8}
Let $f_n,\rho_n,F_n$  be defined as in \eqref{inductionD} and  $[X_n,V_n,\Omega_n,H_n]$ be the solution to \eqref{ptn}. If $0<2\epsilon_0<\epsilon$ and \[\Omega_n(0;0,z)=\omega \in [0+\epsilon_0, \epsilon-\epsilon_0]\quad  \mbox{and}\quad H_n(0;0,z)=\eta\in [-R,R],\]

then
on $0\leq t\leq t_0$,
\begin{eqnarray*}
|f_{n+1}(t)-f_{n}(t)|_{L^\infty_z}\leq C(\mathring{f},\epsilon_0)\int_0^t |f_n(\tau)-f_{n-1}(\tau)|_{L^\infty_z}  d\tau,
\end{eqnarray*}
where $t_0>0$ is a small number defined as in Proposition \ref{prop 4.3}.
\end{lemma}

\begin{proof}
 Consider
\[|f_{n+1}(t,z)-f_{n}(t,z)|=|\mathring{f}(Z_{n}(0;t,z))-\mathring{f}(Z_{n-1}(0;t,z))|\leq \|\partial\mathring{f}\|_{L^\infty_z} |Z_{n}(0;t,z)-Z_{n-1}(0;t,z)|.\]

We have
\begin{eqnarray}\label{eq1}
|X_n(s;t,z)-X_{n-1}(s;t,z)|\leq \int_s^t|V_n(\tau;t,z)-V_{n-1}(\tau;t,z)|d\tau,
\end{eqnarray}
and
\begin{align}
\begin{aligned}\nonumber
&\hspace{-3em}|V_n(s;t,z)-V_{n-1}(s;t,z)|\\\leq & \int_s^t|F^+_n(\tau,X_n(\tau;t,z),\Omega_n(\tau;t,z))-F^+_{n-1}(\tau,X_{n-1}(\tau;t,z),\Omega_{n-1}(\tau;t,z))|d\tau\\
=&\int_s^t|F^+_n(\tau,X_n(\tau;t,z),\Omega_n(\tau;t,z))
-F^+_{n-1}(\tau,X_{n}(\tau;t,z),\Omega_{n}(\tau;t,z))
\\&+F^+_{n-1}(\tau,X_{n}(\tau;t,z),\Omega_{n}(\tau;t,z))
-F^+_{n-1}(\tau,X_{n}(\tau;t,z),\Omega_{n-1}(\tau;t,z))
\\&+F^+_{n-1}(\tau,X_{n}(\tau;t,z),\Omega_{n-1}(\tau;t,z))
-F^+_{n-1}(\tau,X_{n-1}(\tau;t,z),\Omega_{n-1}(\tau;t,z))|d\tau
\\
\leq&\int_s^t|F^+_n(\tau,X_n(\tau;t,z),\Omega_n(\tau;t,z))
-F^+_{n-1}(\tau,X_{n}(\tau;t,z),\Omega_{n}(\tau;t,z))|d\tau
\\&+\int_s^t|F^+_{n-1}(\tau,X_{n}(\tau;t,z),\Omega_{n}(\tau;t,z))
-F^+_{n-1}(\tau,X_{n}(\tau;t,z),\Omega_{n-1}(\tau;t,z))|d\tau
\\&+\int_s^t|F^+_{n-1}(\tau,X_{n}(\tau;t,z),\Omega_{n-1}(\tau;t,z))
-F^+_{n-1}(\tau,X_{n-1}(\tau;t,z),\Omega_{n-1}(\tau;t,z))|d\tau
\\=&:I_1+I_2+I_3.
\end{aligned}
\end{align}

We can use Lemma \ref{pFb} to obtain
\begin{eqnarray*}
I_2&\leq&C\int_s^t |\Omega_{n}(\tau;t,z))
-\Omega_{n-1}(\tau;t,z)|d\tau,
\end{eqnarray*}
and
\begin{eqnarray*}
I_3&\leq&C\int_s^t |X_{n}(\tau;t,z)
-X_{n-1}(\tau;t,z)|d\tau.
\end{eqnarray*}
This yields the following estimate:
\begin{align}
\begin{aligned}\label{eq2}
&\hspace{-3em}|V_n(s;t,z)-V_{n-1}(s;t,z)|\\
\leq&\int_s^t|F^+_n(\tau,X_n(\tau;t,z),\Omega_n(\tau;t,z))
-F^+_{n-1}(\tau,X_{n}(\tau;t,z),\Omega_{n}(\tau;t,z))|d\tau
\\&+C\int_s^t |\Omega_{n}(\tau;t,z))
-\Omega_{n-1}(\tau;t,z)|d\tau
+C\int_s^t |X_{n}(\tau;t,z)
-X_{n-1}(\tau;t,z)|d\tau.
\end{aligned}
\end{align}

Similarly, we have
\begin{eqnarray}\label{eq3}|\Omega_n(s;t,z)-\Omega_{n-1}(s;t,z)|\leq \int_s^t|H_n(\tau;t,z)-H_{n-1}(\tau;t,z)|d\tau,\end{eqnarray}
and
\begin{align}
\begin{aligned}\nonumber
&\hspace{-3em}|H_n(s;t,z)-H_{n-1}(s;t,z)|\\\leq& \int_s^t|F^-_n(\tau,X_n(\tau;t,z),\Omega_n(\tau;t,z))-F^-_{n-1}(\tau,X_{n-1}(\tau;t,z),\Omega_{n-1}(\tau;t,z))\\
&
+F^h(\Omega_n(s;t,z))- F^h(\Omega_{n-1}(s;t,z))|d\tau
\\
\leq & \int_s^t|F^-_n(\tau,X_n(\tau;t,z),\Omega_n(\tau;t,z))-F^-_{n-1}(\tau,X_{n-1}(\tau;t,z),\Omega_{n-1}(\tau;t,z))|d\tau \\
&
+\int_s^t|F^h(\Omega_n(s;t,z))- F^h(\Omega_{n-1}(s;t,z))|d\tau\\:=&J_1+J_2.
\end{aligned}
\end{align}

From the result in Lemma \ref{pFb}, it follows that
\begin{align}
\begin{aligned}\nonumber
J_1
=&\int_s^t|F^-_n(\tau,X_n(\tau;t,z),\Omega_n(\tau;t,z))
-F^-_{n-1}(\tau,X_{n}(\tau;t,z),\Omega_{n}(\tau;t,z))
\\&+F^-_{n-1}(\tau,X_{n}(\tau;t,z),\Omega_{n}(\tau;t,z))
-F^-_{n-1}(\tau,X_{n}(\tau;t,z),\Omega_{n-1}(\tau;t,z))
\\&+F^-_{n-1}(\tau,X_{n}(\tau;t,z),\Omega_{n-1}(\tau;t,z))
-F^-_{n-1}(\tau,X_{n-1}(\tau;t,z),\Omega_{n-1}(\tau;t,z))|d\tau
\\
\leq&\int_s^t|F^-_n(\tau,X_n(\tau;t,z),\Omega_n(\tau;t,z))
-F^-_{n-1}(\tau,X_{n}(\tau;t,z),\Omega_{n}(\tau;t,z))|d\tau
\\&+\int_s^t|F^-_{n-1}(\tau,X_{n}(\tau;t,z),\Omega_{n}(\tau;t,z))
-F^-_{n-1}(\tau,X_{n}(\tau;t,z),\Omega_{n-1}(\tau;t,z))|d\tau
\\&+\int_s^t|F^-_{n-1}(\tau,X_{n}(\tau;t,z),\Omega_{n-1}(\tau;t,z))
-F^-_{n-1}(\tau,X_{n-1}(\tau;t,z),\Omega_{n-1}(\tau;t,z))|d\tau
\\
\leq&\int_s^t|F^-_n(\tau,X_n(\tau;t,z),\Omega_n(\tau;t,z))
-F^-_{n-1}(\tau,X_{n}(\tau;t,z),\Omega_{n}(\tau;t,z))|d\tau
\\&+C\int_s^t |\Omega_{n}(\tau;t,z))
-\Omega_{n-1}(\tau;t,z)|d\tau
+C\int_s^t |X_{n}(\tau;t,z)
-X_{n-1}(\tau;t,z)|d\tau.
\end{aligned}
\end{align}

Since $\epsilon_0\leq\Omega_n(s;t,z)\leq \epsilon-\epsilon_0$ on $0\leq s\leq t\leq t_0$,
we have
\begin{eqnarray*}
J_2&=& \int_s^t|F^h(\Omega_n(\tau;t,z))- F^h(\Omega_{n-1}(\tau;t,z))|d\tau\\
&\leq& C\int_s^t|\Omega_n(\tau;t,z)-\Omega_{n-1}(\tau;t,z)|d\tau.
\end{eqnarray*}

Therefore, we have

\begin{align}
\begin{aligned}\label{eq4}
&\hspace{-3em}|H_n(s;t,z)-H_{n-1}(s;t,z)|\\\leq&
\int_s^t|F^-_n(\tau,X_n(\tau;t,z),\Omega_n(\tau;t,z))
-F^-_{n-1}(\tau,X_{n}(\tau;t,z),\Omega_{n}(\tau;t,z))|d\tau
\\&+C\int_s^t |\Omega_{n}(\tau;t,z))
-\Omega_{n-1}(\tau;t,z)|d\tau
+C\int_s^t |X_{n}(\tau;t,z)
-X_{n-1}(\tau;t,z)|d\tau.
\end{aligned}
\end{align}
We summarize the above argument of \eqref{eq1}-\eqref{eq4} as follows.\phantom{\eqref{eq2}\eqref{eq3}}

If $0\leq t\leq t_0$, then
\begin{eqnarray*}
&&\hspace{-3em} |X_n(0;t,z)-X_{n-1}(0;t,z)|+|V_n(0;t,z)-V_{n-1}(0;t,z)|\\
&&+|\Omega_n(0;t,z)-\Omega_{n-1}(0;t,z)|+|H_n(0;t,z)-H_{n-1}(0;t,z)|\\
&\leq&\int_0^t|F^+_n(\tau,X_n(\tau;t,z),\Omega_n(\tau;t,z))
-F^+_{n-1}(\tau,X_{n}(\tau;t,z),\Omega_{n}(\tau;t,z))|d\tau\\
&&+
\int_0^t|F^-_n(\tau,X_n(\tau;t,z),\Omega_n(\tau;t,z))
-F^-_{n-1}(\tau,X_{n}(\tau;t,z),\Omega_{n}(\tau;t,z))|d\tau
\\
 &&+2C\int_0^t |X_{n}(\tau;t,z)
-X_{n-1}(\tau;t,z)|d\tau+\int_0^t|V_n(\tau;t,z)-V_{n-1}(\tau;t,z)|d\tau
\\&&+2C\int_0^t |\Omega_{n}(\tau;t,z))
-\Omega_{n-1}(\tau;t,z)|d\tau
+\int_0^t|H_n(\tau;t,z)-H_{n-1}(\tau;t,z)|d\tau.
\end{eqnarray*}
By  Granwall's lemma, we have
\begin{eqnarray*}
&&\hspace{-3em} |X_n(0;t,z)-X_{n-1}(0;t,z)|+|V_n(0;t,z)-V_{n-1}(0;t,z)|\\
&&+|\Omega_n(0;t,z)-\Omega_{n-1}(0;t,z)|+|H_n(0;t,z)-H_{n-1}(0;t,z)|\\
&\leq&e^{(C+1)t_0}\int_0^t|F^+_n(\tau,X_n(\tau;t,z),\Omega_n(\tau;t,z))
-F^+_{n-1}(\tau,X_{n}(\tau;t,z),\Omega_{n}(\tau;t,z))|d\tau\\
&&+
e^{(C+1)t_0}\int_0^t|F^-_n(\tau,X_n(\tau;t,z),\Omega_n(\tau;t,z))
-F^-_{n-1}(\tau,X_{n}(\tau;t,z),\Omega_{n}(\tau;t,z))|d\tau.
\end{eqnarray*}

Note that
\begin{eqnarray*}
&&\hspace{-3em}F^\pm_n(\tau,X_n(\tau;t,z),\Omega_n(\tau;t,z))
-F^\pm_{n-1}(\tau,X_{n}(\tau;t,z),\Omega_{n}(\tau;t,z))\\
&=&\Big[F_n(\tau,X_n(\tau;t,z)+\Omega_n(\tau;t,z))
-F_{n-1}(\tau,X_{n}(\tau;t,z)+\Omega_{n}(\tau;t,z))\Big]\\
&&\pm\Big[F_n(\tau,X_n(\tau;t,z)-\Omega_n(\tau;t,z))
-F_{n-1}(\tau,X_{n}(\tau;t,z)-\Omega_{n}(\tau;t,z))\Big]
\end{eqnarray*}
and
\begin{eqnarray*}
&&\hspace{-3em}|F_n(\tau,x)
-F_{n-1}(\tau,x)|\\
&=&\bigg|\frac{\int_{-\infty}^x\rho_n(\tau,y)dy-\int^{\infty}_x\rho_n(\tau,y)dy}{2}-\frac{\int_{-\infty}^x\rho_{n-1}(\tau,y)dy-\int^{\infty}_x\rho_{n-1}(\tau,y)dy}{2}\bigg|
\\
&\leq& \frac{1}{2}\bigg| \int_{-\infty}^x \big[\rho_n(\tau,y)-\rho_{n-1}(\tau,y) \big]dy-\int^{\infty}_x \big[\rho_n(\tau,y)-\rho_{n-1}(\tau,y)\big]dy\bigg|
\\&\leq&\|\rho_n(\tau)-\rho_{n-1}(\tau)\|_{L^1_x}.
\end{eqnarray*}

Thus, we have
\begin{eqnarray*}
&&\hspace{-3em} |X_n(0;t,z)-X_{n-1}(0;t,z)|+|V_n(0;t,z)-V_{n-1}(0;t,z)|\\
&&+|\Omega_n(0;t,z)-\Omega_{n-1}(0;t,z)|+|H_n(0;t,z)-H_{n-1}(0;t,z)|\\
&\leq&4e^{(C+1)t_0}\int_0^t \|\rho_n(\tau)-\rho_{n-1}(\tau)\|_{L^1_x}  d\tau.
\end{eqnarray*}

This implies that
\begin{eqnarray*}
|Z_n(0;t,z)-Z_{n-1}(0;t,z)|&\leq &4e^{(C+1)t_0}\int_0^t \|\rho_n(\tau)-\rho_{n-1}(\tau)\|_{L^1_x}  d\tau
\\&\leq &4e^{(C+1)t_0}P_n^x(t_0)\int_0^t \|\rho_n(\tau)-\rho_{n-1}(\tau)\|_{L^\infty_x}  d\tau\\
&\leq &8\epsilon e^{(C+1)t_0}P_n^x(t_0)P_n^v(t_0)P_n^\eta(t_0) \int_0^t \|f_n(\tau)-f_{n-1}(\tau)\|_{L^\infty_z} d\tau.
\end{eqnarray*}

Thus, we have
\begin{eqnarray*}
|f_{n+1}(t,z)-f_{n}(t,z)|&\leq &\|\partial\mathring{f}\|_{L^\infty_z} |Z_{n}(0;t,z)-Z_{n-1}(0;t,z)|\\
&\leq& \|\partial\mathring{f}\|_{L^\infty_z}8\epsilon e^{(C+1)t_0}P^x(t_0)P^v(t_0)P^\eta(t_0) \int_0^t \|f_n(\tau)-f_{n-1}(\tau)\|_{L^\infty_z} d\tau.
\end{eqnarray*}
By Lemma \ref{pFb},
\[P_n^x(t_0)P_n^v(t_0)P_n^\eta(t_0) \leq C(\mathring{f},\epsilon_0).\]
Finally, we summarize the above calculations as follows.

If $0\leq t\leq t_0$, then
\begin{eqnarray*}
|f_{n+1}(t)-f_{n}(t)|_{L^\infty_z}\leq C(\mathring{f},\epsilon_0)\int_0^t |f_n(\tau)-f_{n-1}(\tau)|_{L^\infty_z}  d\tau.
\end{eqnarray*}
\end{proof}
\begin{lemma}\label{lemma 4.9}
For a given $0<2\epsilon_0<\epsilon$, we assume that
\[\Omega_n(0;0,z)=\omega \in [0+\epsilon_0, \epsilon-\epsilon_0]\quad  \mbox{and}\quad H_n(0;0,z)=\eta\in [-R,R].\]

Let $[X_n,V_n,\Omega_n,H_n]$ be the particle trajectory that satisfies \eqref{ptn} with $f_n,\rho_n,F_n$ in \eqref{inductionD}. Then
on $0\leq t\leq t_0$, we have
\begin{eqnarray*}
&&\hspace{-2em}|\partial X_n(s;t,z)-\partial X_{n-1}(s;t,z)|+|\partial V_n(s;t,z)-\partial V_{n-1}(s;t,z)|
\\&&\hspace{-1.5em}
+|\partial \Omega_n(s;t,z)-\partial \Omega_{n-1}(s;t,z)|+|\partial H_n(s;t,z)-\partial H_{n-1}(s;t,z)|
\\&\leq &C \int_s^t\|\rho_n(\tau)-\rho_{n-1}(\tau)\|_{L^\infty_x}d\tau+C\int_s^t|X_n(\tau;t,z)-X_{n-1}(\tau;t,z)|d\tau\\&&+C\int_s^t|\Omega_n(\tau;t,z)-\Omega_{n-1}(\tau;t,z)|d\tau,
\end{eqnarray*}

where $t_0$ is a small number defined as in Proposition \ref{prop 4.3}.
\end{lemma}
\begin{proof}
Let $\partial$ be one of differential operators $\big\{\frac{\partial}{\partial x},\frac{\partial}{\partial v},\frac{\partial}{\partial \omega},\frac{\partial}{\partial \eta}\big\}
$.
Then by  definition,
\begin{eqnarray*}
|\partial X_n(s;t,z)-\partial X_{n-1}(s;t,z)|\leq \int_s^t|\partial V_n(\tau;t,z)-\partial V_{n-1}(\tau;t,z)|d\tau,
\end{eqnarray*}
\begin{eqnarray*}
&&\hspace{-3em}|\partial V_n(s;t,z)-\partial V_{n-1}(s;t,z)|\\&\leq& \int_s^t|\partial_x F^+_n(\tau,X_n(\tau;t,z),\Omega_n(\tau;t,z))\partial X_n(\tau;t,z)+\partial_\omega F^+_n(\tau,X_n(\tau;t,z),\Omega_n(\tau;t,z))\partial \Omega_n(\tau;t,z)
\\&&-\partial_xF^+_{n-1}(\tau,X_{n-1}(\tau;t,z)+\Omega_{n-1}(\tau;t,z))\partial X_{n-1}(\tau;t,z)
\\&&-\partial_\omega F^+_{n-1}(\tau,X_{n-1}(\tau;t,z)+\Omega_{n-1}(\tau;t,z))\partial\Omega_{n-1}(\tau;t,z)|d\tau
\\&\leq&
 \int_s^t|\partial_x F^+_n(\tau,X_n(\tau;t,z),\Omega_n(\tau;t,z))\partial X_n(\tau;t,z)\\&&-\partial_xF^+_{n-1}(\tau,X_{n-1}(\tau;t,z)+\Omega_{n-1}(\tau;t,z))\partial X_{n-1}(\tau;t,z)|d\tau
 \\
 &&
 +\int_s^t|\partial_\omega F^+_n(\tau,X_n(\tau;t,z),\Omega_n(\tau;t,z))\partial \Omega_n(\tau;t,z)\\&&
-\partial_\omega F^+_{n-1}(\tau,X_{n-1}(\tau;t,z)+\Omega_{n-1}(\tau;t,z))\partial\Omega_{n-1}(\tau;t,z)|d\tau
\\&:=&K_1+K_2.
\end{eqnarray*}
For $K_1$, we have
\begin{eqnarray*}
K_1&= &\int_s^t|\partial_x F^+_n(\tau,X_n(\tau;t,z),\Omega_n(\tau;t,z))\partial X_n(\tau;t,z)\\&&-\partial_xF^+_{n-1}(\tau,X_{n-1}(\tau;t,z)+\Omega_{n-1}(\tau;t,z))\partial X_{n-1}(\tau;t,z)|d\tau\\
&=&\int_s^t|\partial_x F^+_n(\tau,X_n(\tau;t,z),\Omega_n(\tau;t,z))\partial X_n(\tau;t,z)
\\&&\hspace{2em}-\partial_x F^+_{n-1}(\tau,X_n(\tau;t,z),\Omega_n(\tau;t,z))\partial X_n(\tau;t,z)
\\
&&\hspace{3em}
+\partial_x F^+_{n-1}(\tau,X_n(\tau;t,z),\Omega_n(\tau;t,z))\partial X_n(\tau;t,z)\\
&&\hspace{2em}
-\partial_x F^+_{n-1}(\tau,X_{n-1}(\tau;t,z),\Omega_n(\tau;t,z))\partial X_n(\tau;t,z)
\\
&&\hspace{3em}
+\partial_x F^+_{n-1}(\tau,X_{n-1}(\tau;t,z),\Omega_n(\tau;t,z))\partial X_n(\tau;t,z)\\
&&\hspace{2em}
-\partial_x F^+_{n-1}(\tau,X_{n-1}(\tau;t,z),\Omega_{n-1}(\tau;t,z))\partial X_n(\tau;t,z)
\\
&&\hspace{3em}
+\partial_x F^+_{n-1}(\tau,X_{n-1}(\tau;t,z),\Omega_{n-1}(\tau;t,z))\partial X_n(\tau;t,z)
\\&&\hspace{2em}-\partial_xF^+_{n-1}(\tau,X_{n-1}(\tau;t,z)+\Omega_{n-1}(\tau;t,z))\partial X_{n-1}(\tau;t,z)|d\tau.\end{eqnarray*}

Therefore,
\begin{eqnarray*}
K_1&\leq &
\int_s^t|\partial_x F^+_n(\tau,X_n(\tau;t,z),\Omega_n(\tau;t,z))\partial X_n(\tau;t,z)\\&&
-\partial_x F^+_{n-1}(\tau,X_n(\tau;t,z),\Omega_n(\tau;t,z))\partial X_n(\tau;t,z)|d\tau
\\
&&
+\int_s^t|\partial_x F^+_{n-1}(\tau,X_n(\tau;t,z),\Omega_n(\tau;t,z))\partial X_n(\tau;t,z)
\\&&-\partial_x F^+_{n-1}(\tau,X_{n-1}(\tau;t,z),\Omega_n(\tau;t,z))\partial X_n(\tau;t,z)|d\tau
\\
&&
+\int_s^t|\partial_x F^+_{n-1}(\tau,X_{n-1}(\tau;t,z),\Omega_n(\tau;t,z))\partial X_n(\tau;t,z)
\\&&-\partial_x F^+_{n-1}(\tau,X_{n-1}(\tau;t,z),\Omega_{n-1}(\tau;t,z))\partial X_n(\tau;t,z)|d\tau
\\
&&
+\int_s^t|\partial_x F^+_{n-1}(\tau,X_{n-1}(\tau;t,z),\Omega_{n-1}(\tau;t,z))\partial X_n(\tau;t,z)
\\&&-\partial_xF^+_{n-1}(\tau,X_{n-1}(\tau;t,z)+\Omega_{n-1}(\tau;t,z))\partial X_{n-1}(\tau;t,z)|d\tau
\\&:=&K_{11}+K_{12}+K_{13}+K_{14}.
\end{eqnarray*}

By elementary calculation, we have
\begin{eqnarray*}
K_{11}&=& \int_s^t|\partial_x F^+_n(\tau,X_n(\tau;t,z),\Omega_n(\tau;t,z))
-\partial_x F^+_{n-1}(\tau,X_n(\tau;t,z),\Omega_n(\tau;t,z))|~|\partial X_n(\tau;t,z)|d\tau\\
&\leq& \int_s^t|\partial_x F^+_n(\tau,X_n(\tau;t,z),\Omega_n(\tau;t,z))
-\partial_x F^+_{n-1}(\tau,X_n(\tau;t,z),\Omega_n(\tau;t,z))|~|\partial X_n(\tau;t,z)|d\tau
\\&\leq& C\int_s^t|\partial_x F^+_n(\tau,X_n(\tau;t,z),\Omega_n(\tau;t,z))
-\partial_x F^+_{n-1}(\tau,X_n(\tau;t,z),\Omega_n(\tau;t,z))|d\tau.
\end{eqnarray*}
Here we used the result in Lemma \ref{pb}.
Note that
\begin{eqnarray*}
\partial_x F_n(t,x)&=&-\rho_n(t,x),\\
F^{\pm}_n(t,x,\omega)&=&F_n(t,x+ \omega)\pm F_n(t,x- \omega).
\end{eqnarray*}
Therefore,
\[\frac{\partial}{\partial x }F^{\pm}_n=-\rho_n(t,x+\omega)\mp\rho_n(t,x-\omega),\]
and this implies that
\begin{eqnarray*}
K_{11}&\leq& C \int_s^t\|\rho_n(\tau)-\rho_{n-1}(\tau)\|_{L^\infty_x}d\tau.
\end{eqnarray*}
For $K_{12}$,
\begin{eqnarray*}
&&\hspace{-1em}K_{12}\\
&&=\int_s^t|\partial_x F^+_{n-1}(\tau,X_n(\tau;t,z),\Omega_n(\tau;t,z))
-\partial_x F^+_{n-1}(\tau,X_{n-1}(\tau;t,z),\Omega_n(\tau;t,z))|~|\partial X_n(\tau;t,z)|d\tau\\
&&\leq C\int_s^t|\partial_x F^+_{n-1}(\tau,X_n(\tau;t,z),\Omega_n(\tau;t,z))
-\partial_x F^+_{n-1}(\tau,X_{n-1}(\tau;t,z),\Omega_n(\tau;t,z))|d\tau\\
&&\leq C\int_s^t\|\partial_x\partial_x F^+_{n-1}\|_{L^\infty_{x,\omega}} |X_n(\tau;t,z)-X_{n-1}(\tau;t,z)|d\tau.
\end{eqnarray*}
By Lemma \ref{pb}, we have
\begin{eqnarray*}
K_{12}&\leq&C\int_s^t|X_n(\tau;t,z)-X_{n-1}(\tau;t,z)|d\tau.
\end{eqnarray*}
For $K_{13}$, similarly,
\begin{eqnarray*}
K_{13}&\leq&C\int_s^t|\Omega_n(\tau;t,z)-\Omega_{n-1}(\tau;t,z)|d\tau.
\end{eqnarray*}
Clearly,
\begin{eqnarray*}
K_{14}&\leq&C\int_s^t|\partial X_n(\tau;t,z)-\partial X_{n-1}(\tau;t,z)|d\tau.
\end{eqnarray*}
Therefore,
\begin{eqnarray*}
K_1&\leq& C \int_s^t\|\rho_n(\tau)-\rho_{n-1}(\tau)\|_{L^\infty_x}d\tau+C\int_s^t|X_n(\tau;t,z)-X_{n-1}(\tau;t,z)|d\tau\\
&&+C\int_s^t|\Omega_n(\tau;t,z)-\Omega_{n-1}(\tau;t,z)|d\tau+
C\int_s^t|\partial X_n(\tau;t,z)-\partial X_{n-1}(\tau;t,z)|d\tau.
\end{eqnarray*}
By a similar method,
\begin{eqnarray*}
K_2&\leq& C \int_s^t\|\rho_n(\tau)-\rho_{n-1}(\tau)\|_{L^\infty_x}d\tau+C\int_s^t|X_n(\tau;t,z)-X_{n-1}(\tau;t,z)|d\tau\\
&&+C\int_s^t|\Omega_n(\tau;t,z)-\Omega_{n-1}(\tau;t,z)|d\tau+
C\int_s^t|\partial \Omega_n(\tau;t,z)-\partial \Omega_{n-1}(\tau;t,z)|d\tau.
\end{eqnarray*}
Thus, we have
\begin{eqnarray*}
&&\hspace{-3em}|\partial V_n(s;t,z)-\partial V_{n-1}(s;t,z)|\\
&\leq& C  \int_s^t\|\rho_n(\tau)-\rho_{n-1}(\tau)\|_{L^\infty_x}d\tau+C\int_s^t|X_n(\tau;t,z)-X_{n-1}(\tau;t,z)|d\tau\\
&&+C\int_s^t|\Omega_n(\tau;t,z)-\Omega_{n-1}(\tau;t,z)|d\tau+
C\int_s^t|\partial X_n(\tau;t,z)-\partial X_{n-1}(\tau;t,z)|d\tau
\\&& +
C\int_s^t|\partial \Omega_n(\tau;t,z)-\partial \Omega_{n-1}(\tau;t,z)|d\tau
.\end{eqnarray*}

Similarly, we have
\begin{eqnarray*}|\partial\Omega_n(s;t,z)-\partial\Omega_{n-1}(s;t,z)|\leq \int_s^t|\partial H_n(\tau;t,z)-\partial H_{n-1}(\tau;t,z)|d\tau,\end{eqnarray*}
and
\begin{eqnarray*}
&&\hspace{-3em}|\partial H_n(s;t,z)-\partial H_{n-1}(s;t,z)|\\&\leq&   \int_s^t|\partial_x F^-_n(\tau,X_n(\tau;t,z),\Omega_n(\tau;t,z))\partial X_n(\tau;t,z)\\
&&+\partial_\omega F^-_n(\tau,X_n(\tau;t,z),\Omega_n(\tau;t,z))\partial \Omega_n(\tau;t,z)
\\&&-\partial_xF^-_{n-1}(\tau,X_{n-1}(\tau;t,z)+\Omega_{n-1}(\tau;t,z))\partial X_{n-1}(\tau;t,z)\\&&-\partial_\omega F^-_{n-1}(\tau,X_{n-1}(\tau;t,z)+\Omega_{n-1}(\tau;t,z))\partial\Omega_{n-1}(\tau;t,z)|d\tau
\\
&&
+\int_s^t|\frac{d}{dx}F^h(\Omega_n(s;t,z))\partial \Omega_n(s;t,z)- \frac{d}{dx}F^h(\Omega_{n-1}(s;t,z))\partial \Omega_{n-1}(s;t,z)|d\tau
\\
&:=&L_1+L_2.\end{eqnarray*}

By the same argument, we have
\begin{eqnarray*}
L_1&\leq& C \int_s^t\|\rho_n(\tau)-\rho_{n-1}(\tau)\|_{L^\infty_x}d\tau+C\int_s^t|X_n(\tau;t,z)-X_{n-1}(\tau;t,z)|d\tau\\
&&+C\int_s^t|\Omega_n(\tau;t,z)-\Omega_{n-1}(\tau;t,z)|d\tau+
C\int_s^t|\partial X_n(\tau;t,z)-\partial X_{n-1}(\tau;t,z)|d\tau \\&&+
C\int_s^t|\partial \Omega_n(\tau;t,z)-\partial \Omega_{n-1}(\tau;t,z)|d\tau
.\end{eqnarray*}
For $L_2$, we have

\begin{eqnarray*}
L_2&= &\int_s^t|\frac{d}{dx}F^h(\Omega_n(s;t,z))\partial \Omega_n(s;t,z)- \frac{d}{dx}F^h(\Omega_{n-1}(s;t,z))\partial \Omega_{n-1}(s;t,z)|d\tau
\\
&= &\int_s^t|\frac{d}{dx}F^h(\Omega_n(s;t,z))\partial \Omega_n(s;t,z)-\frac{d}{dx}F^h(\Omega_n(s;t,z))\partial \Omega_{n-1}(s;t,z)
\\&&+\frac{d}{dx}F^h(\Omega_n(s;t,z))\partial \Omega_{n-1}(s;t,z)- \frac{d}{dx}F^h(\Omega_{n-1}(s;t,z))\partial \Omega_{n-1}(s;t,z)|d\tau
\\
&\leq &\int_s^t|\frac{d}{dx}F^h(\Omega_n(s;t,z))\partial \Omega_n(s;t,z)-\frac{d}{dx}F^h(\Omega_n(s;t,z))\partial \Omega_{n-1}(s;t,z)|d\tau
\\
&&
+\int_s^t|\frac{d}{dx}F^h(\Omega_n(s;t,z))\partial \Omega_{n-1}(s;t,z)- \frac{d}{dx}F^h(\Omega_{n-1}(s;t,z))\partial \Omega_{n-1}(s;t,z)|d\tau\\
&\leq &\int_s^t|\frac{d}{dx} F^h(\Omega_n(s;t,z))|~|\partial \Omega_n(s;t,z)-\partial\Omega_{n-1}(s;t,z)|d\tau
\\
&&
+\int_s^t|\frac{d}{dx}F^h(\Omega_n(s;t,z))- \frac{d}{dx}F^h(\Omega_{n-1}(s;t,z))|~| \partial\Omega_{n-1}(s;t,z)|d\tau.
\end{eqnarray*}
By Lemma \ref{pb} and Proposition \ref{prop 4.3},

\begin{eqnarray*}
L_2&\leq &C\int_s^t~|\partial \Omega_n(s;t,z)-\partial\Omega_{n-1}(s;t,z)|d\tau+C\int_s^t|\Omega_n(s;t,z))- \Omega_{n-1}(s;t,z)|d\tau
.
\end{eqnarray*}

Thus, we have
\begin{eqnarray*}
&&\hspace{-3em}|\partial H_n(s;t,z)-\partial H_{n-1}(s;t,z)|\\
&\leq & C \int_s^t\|\rho_n(\tau)-\rho_{n-1}(\tau)\|_{L^\infty_x}d\tau+C\int_s^t|X_n(\tau;t,z)-X_{n-1}(\tau;t,z)|d\tau\\
&&+C\int_s^t|\Omega_n(\tau;t,z)-\Omega_{n-1}(\tau;t,z)|d\tau+
C\int_s^t|\partial X_n(\tau;t,z)-\partial X_{n-1}(\tau;t,z)|d\tau \\&&+
C\int_s^t|\partial \Omega_n(\tau;t,z)-\partial \Omega_{n-1}(\tau;t,z)|d\tau.
\end{eqnarray*}

Finally, Granwall's lemma implies the desired result.
\end{proof}

We are ready to prove  that $(f_n)$ converges uniformly to some $f$ on $[0,t_0]\times \bbr\times \bbr\times (0,\epsilon)\times\bbr$.

Lemma \ref{lemma 4.8} and mathematical induction yield that
\begin{eqnarray*}
\|f_{n+1}(t)-f_{n}(t)\|_{L^\infty_z}\leq \sup_{0\leq t\leq t_0} \|f_1(t)-f_{0}(t)\|_{L^\infty_z}\frac{C(\mathring{f})^nt^n}{n!}.
\end{eqnarray*}

Therefore, $f_n$ is an uniformly Cauchy sequence. By  definition, $\rho_n$ and $F_n$ are also uniformly Cauchy. By Lemma \ref{pb}, Lemma \ref{lemma 4.9}, equation \eqref{Frho}, and the previous argument, $\partial Z_n$ converges uniformly. Therefore, $f$ is the classical solution to the main equation. We can also obtain uniqueness and continuation criterion by standard argument.

\begin{theorem}\label{local}Let $\mathcal{D}=\bbr\times\bbr\times(0,\epsilon)\times\bbr$ and  $\mathring{f}\in C^1_c(\mathcal{D})$ be a  nonnegative function. Then there exists $t_0>0$ such that the unique classical solution $f(t,x,v,\omega,\eta)\in C^1((0,t_0)\times\mathcal{D})$ to \eqref{pVP} exists. Moreover, $f(t)$ has a compact support for all $t\in (t,t_0)$ and $f(t)\geq0$.  If $(0,t_0)$ is the maximal existence interval and  if
\[ \sup \{|x|+|v|: (x,v,\omega,\eta)\in supp~f(t),~ t\in (0,t_0)\}<\infty,\]
\[\inf\{\omega: (x,v,\omega,\eta)\in supp~f(t),~ t\in (0,t_0)\}>0,\]
\[\sup\{\omega: (x,v,\omega,\eta)\in supp~f(t),~ t\in (0,t_0)\}<\epsilon,\]
and
\[\sup\{|\eta|: (x,v,\omega,\eta)\in supp~f(t),~ t\in (0,t_0)\}<\infty,\]

then the solution is global.
\end{theorem}

\section{Estimates for particle trajectory}\label{sec4}
\setcounter{equation}{0}
In this part, we provide upper bounds of the particle speed. Especially, we obtain control of the relative distance of the atoms through a careful analysis of
their oscillatory motions. These estimates  play a crucial role in the proof of the main theorem. Throughout this section, we  assume that  $[X(s), V(s),\Omega(s), H(s)]$ is the particle trajectory that satisfies the following   ODE system:
\begin{align}
\begin{aligned}\label{particle}
\frac{d}{ds}X(s) &= V(s),\\
\frac{d}{ds}V(s)& = F^+ (s,X(s),\Omega(s)),\\
\frac{d}{ds}\Omega(s) &= H(s),\\
 \frac{d}{ds}H(s) &=
F^-(s,X(s),\Omega(s)) +F^h(\Omega(s)),
\end{aligned}
\end{align}
with initial data
\[  X(0) = x\in \bbr \quad \mbox{and} \quad  V(0) = v \in \bbr, \]
and
\[  \Omega(0) = \omega\in (0,\epsilon) \quad \mbox{and} \quad  H(0) = \eta \in \bbr, \]
where  $F^\pm(t, \cdot,\cdot)$ is a given bounded vector field.
Additionally, we assume that $F^{h}$ satisfies the conditions (H1)-(H4) in Section \ref{intro}. Since the estimates for
$X(s)$ and $V(s)$ are standard, we focus on the control of the relative position $\Omega(s)$ and the oscillatory velocity $H(s)$.
We start with a basic energy type identity, which will be crucial throughout this section.
\begin{lemma}\label{lemma 5.1}
Let  $[X(s), V(s),\Omega(s), H(s)]$ be the particle trajectory satisfying  ODE system \eqref{particle}. Then for any $t_2>t_2\geq0$,
\begin{eqnarray*}
\frac{1}{2}|H(t_2)|^2-\frac{1}{2}|H(t_1)|^2
& =&
\int_{t_1}^{t_2} H(s)F^-(s,X(s),\Omega(s))ds +\int_{\Omega(t_1)}^{\Omega(t_2)}F^h(y)dy.
\end{eqnarray*}

\end{lemma}
\begin{proof}
The fourth equation in \eqref{particle} implies
\begin{eqnarray*}
\frac{1}{2}\frac{d}{ds}|H(s)|^2&= &H(s)\frac{d}{ds}H(s)\\
& =&
H(s)F^-(s,X(s),\Omega(s)) +H(s)F^h(\Omega(s)).
\end{eqnarray*}

We integrate it over $[t_1,t_2]$ to obtain
\begin{eqnarray*}
\frac{1}{2}|H(t_2)|^2-\frac{1}{2}|H(t_1)|^2
& =&
\int_{t_1}^{t_2} H(s)F^-(s,X(s),\Omega(s))ds + \int_{t_1}^{t_2}H(s)F^h(\Omega(s))ds.
\end{eqnarray*}

We note that by the change of variables,
\begin{eqnarray*}
\int_{t_1}^{t_2}H(s)F^h(\Omega(s))ds&=&\int_{t_1}^{t_2}\frac{d}{ds}\Omega(s)F^h(\Omega(s))ds\\&=&\int_{\Omega(t_1)}^{\Omega(t_2)}F^h(y)dy.
\end{eqnarray*}

Therefore, we have
\begin{eqnarray*}
\frac{1}{2}|H(t_2)|^2-\frac{1}{2}|H(t_1)|^2
& =&
\int_{t_1}^{t_2} H(s)F^-(s,X(s),\Omega(s))ds +\int_{\Omega(t_1)}^{\Omega(t_2)}F^h(y)dy.
\end{eqnarray*}

\end{proof}

\begin{lemma}\label{lemma 5.2}
Let  $[X(s), V(s),\Omega(s), H(s)]$ be a  particle trajectory  satisfying  ODE system \eqref{particle} and $\Omega(s)\in (0,\epsilon)$.
We assume that there exists a positive constant $C>0$ such that
\[\|F^-(t,\cdot,\cdot)\|_{L^\infty_{x,\omega}}<C,\quad t\geq 0. \]
If $H(t)\geq 0$, $t\in [t_1,t_2]$,
then
\begin{eqnarray*}
\Big|\int_{t_1}^{t_2} H(s)F^-(s,X(s),\Omega(s))ds\Big|& \leq& C\epsilon.
\end{eqnarray*}
Similarly, if  $H(t)\leq 0$, $t\in [t_1,t_2]$,
then
\begin{eqnarray*}
\Big|\int_{t_1}^{t_2} H(s)F^-(s,X(s),\Omega(s))ds\Big|& \leq& C\epsilon.
\end{eqnarray*}

\end{lemma}

\begin{proof}
We assume that
\[H(t)\geq 0,\quad t\in [t_1,t_2].\]

By the assumption, we have
\begin{eqnarray*}
\Big|\int_{t_1}^{t_2} H(s)F^-(s,X(s),\Omega(s))ds\Big|& \leq& \int_{t_1}^{t_2} |H(s)F^-(s,X(s),\Omega(s))|ds\\
& =& \int_{t_1}^{t_2} |H(s)| |F^-(s,X(s),\Omega(s))|ds\\
& \leq&C  \int_{t_1}^{t_2} |H(s)| ds.
\end{eqnarray*}

Since $H(t)\geq 0$,
\begin{eqnarray*}
\Big|\int_{t_1}^{t_2} H(s)F^-(s,X(s),\Omega(s))ds\Big|& \leq&
C  \int_{t_1}^{t_2} H(s) ds\\
& =&C  (\Omega(t_2)-\Omega(t_1))
 =C  |\Omega(t_2)-\Omega(t_1)|
 \\& \leq &C\epsilon.
\end{eqnarray*}
Similarly, for $H(t)\leq 0$, we can obtain
\begin{eqnarray*}
\Big|\int_{t_1}^{t_2} H(s)F^-(s,X(s),\Omega(s))ds\Big|& \leq& C  |\Omega(t_2)-\Omega(t_1)|\leq C\epsilon.
\end{eqnarray*}
\end{proof}

The following two lemmas guarantee the existence of the stopping time $\bar{t}$ after the atom escapes the balancing point of
$F^-$ and $F^h$.
\begin{lemma}\label{lemma 5.4}
We assume that there is $C>0$ such that
\[\|F^-(t,\cdot,\cdot)\|_{L^\infty_{x,\omega}}< C,\quad t\geq 0. \]
Let  $\Omega_M$ be a constant satisfying
\begin{eqnarray}\label{eq 4.16}
 F^h(\Omega_M)=-C,\quad 0<\frac{\epsilon}{2}<\Omega_M<\epsilon.
\end{eqnarray}

We additionally assume that
\[\Omega(t_1)=\Omega_M \]
and  there is $t_1^*>0$ such that
\begin{eqnarray}\label{eq 5.22}
\Omega(t)<\Omega_M,\quad \mbox{on $t\in (t_1^*,t_1)$.}
\end{eqnarray}

Then  $H(t_1)\geq0$. Moreover, if $H(t_1)>0$, then there is $\bar{t}_1>t_1$ such that
\[H(\bar{t}_1)=0\]
 and  $H(t)>0$ on $t\in (t_1,\bar{t}_1)$.

\end{lemma}
\begin{proof}
We assume that
\[H(t_1)<0.\]
Then by  continuity, there is $\delta>0$ such that
$H(t)<0$  on $t\in (t_1-\delta,t_1+\delta)$.

Let $t_m=\max\{t_1-\delta,t_1^*\}$.
Thus,
\[\Omega_M-\Omega(t_m)=\Omega(t_1)-\Omega(t_m)=\int^{t_1}_{t_m}H(s)ds<0.\]
However, this is a contradiction with \eqref{eq 5.22}.
Therefore, $H(t_1)\geq0$.

We assume that $H(t_1)>0$. Let $[t_1,\bar{t}_1)$ is the maximal interval such that
\[H(t)>0,\quad t\in [t_1,\bar{t}_1).\]
Therefore, $\Omega(t)$ is increasing on $[t_1,\bar{t}_1)$. This implies that
\begin{eqnarray}\label{eq 5.3}
\Omega(t)\geq \Omega_M, \quad \mbox{on $[t_1,\bar{t}_1)$.}
\end{eqnarray}
 Then for $t_1<t<s<\bar{t}_1$,
\begin{eqnarray*}
H(s)-H(t)=\int_{t}^s \frac{d}{ds}H(s)  ds=\int_{t}^s [F^-(s,X(s),\Omega(s)) +F^h(\Omega(s))]ds.
\end{eqnarray*}
By   \eqref{eq 4.16}, \eqref{eq 5.3}  and   (H1)-(H4),
\begin{eqnarray*}
H(s)-H(t)<0.
\end{eqnarray*}
Thus, $H(\cdot)$ is monotone decreasing on $[t_1,\bar{t}_1)$.

Now, we assume that $\bar{t}_1=\infty$. Then  the following limits exist.
\[H^\infty=\lim_{t\to \infty}H(t),\]
\[\Omega^\infty=\lim_{t\to \infty}\Omega(t)>\Omega_M.\]
However,
\[\limsup_{t\to\infty}\frac{d}{ds}H(s)=  \limsup_{t\to\infty}   [F^-(s,X(s),\Omega(s)) +F^h(\Omega(s))]< -C+C =0.\]
Thus, $H^\infty=-\infty$. This is a contradiction with $H(t)>0$ at $t\in [t,\bar{t}_1)$. Therefore, $\bar{t}_1$ is finite and we obtain the desired result.

\end{proof}

\begin{lemma}\label{lemma 5.4s}We assume that there is $C>0$ such that
\[\|F^-(t,\cdot,\cdot)\|_{L^\infty_{x,\omega}}< C,\quad t\geq 0. \]
Let $\Omega_m$ be a constant satisfying
\[F^h(\Omega_m)=C,\quad 0<\Omega_m<\frac{\epsilon}{2}<\epsilon. \]

If
\[\Omega(t_1)=\Omega_m \]
and  there is $t_1^*>0$ such that
\begin{eqnarray*}
\Omega(t)>\Omega_m,\quad \mbox{
on $t\in (t_1^*,t_1)$,}
\end{eqnarray*}

then  $H(t_1)\leq0$. Moreover, if $H(t_1)<0$, then there is $\bar{t}_1>t_1$ such that
\[H(\bar{t}_1)=0\]
 and  $H(t)<0$ on $t\in (t_1,\bar{t}_1)$.

\end{lemma}
\begin{proof}
The proof of this lemma is similar to Lemma \ref{lemma 5.4}. We omit the proof.
\end{proof}

In the following lemma, we measure the potential energy of the atom accumulated until the first stopping time $\bar{t}_1$, and establish the existence of returning time $t^r$.
We also obtain a control on the relative velocity between the atoms escape from and return to the non-autonomous region.
\begin{lemma}\label{lemma 5.6}
We assume that there is $C>0$ such that
\[\|F^-(t,\cdot,\cdot)\|_{L^\infty_{x,\omega}}< C,\quad t\geq 0. \]
Let  $\Omega_M$ be a constant satisfying
\[ F^h(\Omega_M)=-C,\quad 0<\frac{\epsilon}{2}<\Omega_M<\epsilon. \]

We assume that for some $t_1>0$,
\[\Omega(t_1)=\Omega_M , \quad H(t_1)>0,\]
and  there is $t_1^*>0$ such that
\begin{eqnarray*}
\Omega(t)<\Omega_M,\quad \mbox{on $t\in (t_1^*,t_1)$.
}
\end{eqnarray*}

Then
\begin{eqnarray}\label{eq 5.111}
\frac{1}{2}|H(t_1)|^2+I_M-\epsilon C\leq \int^{\frac{\epsilon}{2}}_{\Omega(\bar{t}_1)}F^h(y)dy
 \leq
\frac{1}{2}|H(t_1)|^2+I_M+\epsilon C,
\end{eqnarray}
where
\[I_M=\int_{\Omega_M}^{\frac{\epsilon}{2}}F^h(y)dy, \]
and there is $t_1^r>t_1$ such that
\[\Omega(t_1^r)=\Omega_M, \quad H(t_1^r)<0,\]
and
\[|H(t)|
\leq \big(|H(t_1)|^2+4\epsilon C\big)^{\frac{1}{2}},\quad \mbox{for all $t\in (t_1,t_1^r)$.
}\]

\end{lemma}

\begin{proof}
By Lemma \ref{lemma 5.4}, there is $\bar{t}_1>t_1$ such that
\[H(\bar{t}_1)=0\]
 and  $H(t)>0$ on $t\in (t_1,\bar{t}_1)$.

By Lemmas \ref{lemma 5.1} and \ref{lemma 5.2},
\begin{eqnarray*}
\frac{1}{2}|H(\bar{t}_1)|^2+\int^{\frac{\epsilon}{2}}_{\Omega(\bar{t}_1)}F^h(y)dy
& =&
\frac{1}{2}|H(t_1)|^2+\int_{\Omega(t_1)}^{\frac{\epsilon}{2}}F^h(y)dy+\int_{t_1}^{\bar{t}_1} H(s)F^-(s,X(s),\Omega(s))ds
\end{eqnarray*}
and
\begin{eqnarray*}
\Big|\int_{t_1}^{\bar{t}_1} H(s)F^-(s,X(s),\Omega(s))ds\Big|& \leq& C\epsilon.
\end{eqnarray*}

Since we assume that $F^h(x)> 0$ on $\displaystyle 0< x<\frac{\epsilon}{2}$,  $\displaystyle F^h(x)< 0$ on $\displaystyle \frac{\epsilon}{2}< x<\epsilon$,

\[\int^{\frac{\epsilon}{2}}_{\omega}F^h(y)dy
\geq 0,\quad \mbox{for any $\omega\in (0,\epsilon)$.
}\]

Therefore, we have
\begin{eqnarray}\label{eq 5.7}
\frac{1}{2}|H(t_1)|^2+I_M-\epsilon C\leq \int^{\frac{\epsilon}{2}}_{\Omega(\bar{t}_1)}F^h(y)dy
 \leq
\frac{1}{2}|H(t_1)|^2+I_M+\epsilon C,
\end{eqnarray}
where
\[I_M=\int_{\Omega_M}^{\frac{\epsilon}{2}}F^h(y)dy, \quad \mbox{and $\Omega_M=\Omega(t_1)$.
}\]

Clearly, $H(\bar{t}_1)=0$ and $\Omega_M<\Omega(\bar{t}_1)<\epsilon$.  Moreover, $\ddot{\Omega}(t)=\dot{H}(t)<0$, if $\Omega_M<\Omega(\bar{t}_1)<\epsilon$.
Therefore, there is $t_1^r>\bar{t}_1$ such that
\begin{eqnarray}\label{eq 5.5}\Omega(t_1^r)=\Omega_M
\end{eqnarray}
 and
\begin{eqnarray}\label{eq 5.6} H(t)<0,\quad \mbox{on} \quad t\in (\bar{t}_1, t_1^r).  \end{eqnarray}

By Lemma \ref{lemma 5.1}  and \eqref{eq 5.5},
\begin{align}
\begin{aligned}\label{eq 5.77}
\frac{1}{2}|H(t_1^r)|^2+I_M&=
\frac{1}{2}|H(t_1^r)|^2+\int^{\frac{\epsilon}{2}}_{\Omega(t_1^r)}F^h(y)dy
\\& =
\frac{1}{2}|H(\bar{t}_1)|^2+\int_{\Omega(\bar{t}_1)}^{\frac{\epsilon}{2}}F^h(y)dy+\int_{\bar{t}_1}^{t_1^r} H(s)F^-(s,X(s),\Omega(s))ds.
\end{aligned}
\end{align}
By Lemma \ref{lemma 5.2} and \eqref{eq 5.6},
\begin{eqnarray*}
\Big|\int_{\bar{t}_1}^{t_1^r} H(s)F^-(s,X(s),\Omega(s))ds\Big|& \leq& C\epsilon.
\end{eqnarray*}

Therefore, the above inequality, \eqref{eq 5.7} and \eqref{eq 5.77} imply
\begin{eqnarray*}
\frac{1}{2}|H(t_1^r)|^2+I_M&\leq &
\frac{1}{2}|H(\bar{t}_1)|^2+\int_{\Omega(\bar{t}_1)}^{\frac{\epsilon}{2}}F^h(y)dy+\epsilon C\\
&\leq &
\frac{1}{2}|H(\bar{t}_1)|^2+\frac{1}{2}|H(t_1)|^2+I_M+2\epsilon C\\
&=&\frac{1}{2}|H(t_1)|^2+I_M+2\epsilon C.
\end{eqnarray*}
Here we used $H(\bar{t}_1)=0$.

Thus,  we have
\begin{eqnarray*}
|H(t_1^r)|
\leq \big(|H(t_1)|^2+4\epsilon C\big)^{\frac{1}{2}}.
\end{eqnarray*}
Since $H(\cdot)$ is monotone decreasing and negative on $(t_1,t_1^r)$,
$H(t_1^r)<0$ and
\[|H(t)|
\leq \big(|H(t_1)|^2+4\epsilon C\big)^{\frac{1}{2}},\quad \mbox{for all $t\in (t_1,t_1^r)$.
}\]

\end{proof}

\begin{lemma}\label{lemma 5.6s}
We assume that there is $C>0$ such that
\[\|F^-(t,\cdot,\cdot)\|_{L^\infty_{x,\omega}}< C,\quad t\geq 0. \]
Let $\Omega_m$ and $\Omega_M$ be   constants satisfying
\[F^h(\Omega_m)=C,\quad 0<\Omega_m<\frac{\epsilon}{2}<\epsilon. \]

We assume that for some $t_1>0$,
\[\Omega(t_1)=\Omega_m , \quad H(t_1)<0,\]
and  there is $t_1^*>0$ such that
\begin{eqnarray*}
\Omega(t)>\Omega_m
\end{eqnarray*}
on $t\in (t_1^*,t_1)$.

Then
\begin{eqnarray*}
\frac{1}{2}|H(t_1)|^2+I_m-\epsilon C\leq \int^{\frac{\epsilon}{2}}_{\Omega(\bar{t}_1)}F^h(y)dy
 \leq
\frac{1}{2}|H(t_1)|^2+I_m+\epsilon C,
\end{eqnarray*}
where
\[I_m=\int_{\Omega_m}^{\frac{\epsilon}{2}}F^h(y)dy, \]
and there is $t_1^r>t_1$ such that
\[\Omega(t_1^r)=\Omega_m, \quad H(t_1^r)>0,\]
and
\[|H(t)|
\leq \big(|H(t_1)|^2+4\epsilon C\big)^{\frac{1}{2}},\quad \mbox{for all $t\in (t_1,t_1^r)$.
}\]

\end{lemma}
\begin{proof}By using Lemma \ref{lemma 5.4s} with a similar method as in the proof of Lemma \ref{lemma 5.6},
we can easily obtain  this result.
\end{proof}
For the next lemma, we introduce the following integrals related to the oscillatory energy.
\begin{definition}
We define $h_M(x):[\frac{\epsilon}{2},\epsilon)\mapsto [0,\infty)$ by
\[h_M(x)=\int^{\frac{\epsilon}{2}}_xF^h(y)dy,\]

and
 $h_m(x):(0,\frac{\epsilon}{2}]\mapsto [0,\infty)$ by
\[h_m(x)=\int^{\frac{\epsilon}{2}}_xF^h(y)dy.\]

\end{definition}

\begin{remark}
The functions $h_M$ and $h_m$ are monotone and continuous. Therefore, the inverses to $h_M$ and $h_m$ exist and they are monotone and continuous.
\end{remark}

The following lemmas are crucially used to control the large oscillatory velocity in Section \ref{sec5}. (See (\ref{eq 6.3}) and (\ref{f3})).

\begin{lemma}\label{lemma 5.7}
We assume that there is $C>0$ such that
\[\|F^-(t,\cdot,\cdot)\|_{L^\infty_{x,\omega}}< C,\quad t\geq 0. \]
Let  $\Omega_M$ be  a  constant satisfying
\[F^h(\Omega_M)=-C,\quad 0<\frac{\epsilon}{2}<\Omega_M<\epsilon. \]

We assume that for some $t_1>0$,
\[\Omega(t_1)=\Omega_M , \quad H(t_1)>0,\]
and  there is $t_1^*>0$ such that
\begin{eqnarray*}
\Omega(t)<\Omega_M, \quad \mbox{on $t\in (t_1^*,t_1)$.
}
\end{eqnarray*}
Then
\[2\frac{h_M^{-1}\Big(\frac{1}{2}|H(t_1)|^2+I_M-\epsilon C\Big)-\Omega_M}{\sqrt{|H(t_1)|^2+4\epsilon C}}\leq t^r_1- t_1,\]
where $t_1^r$ and $I_M$ are defined as in Lemma \ref{lemma 5.6}.

Moreover, if  $h_M^{-1}\Big(\frac{1}{2}|H(t_1)|^2+I_M-\epsilon C\Big)-\Omega_M>0$, then
\begin{eqnarray*}
\frac{\big||H(t_1^r)|-|H(t_1)|\big|}{t_1^r-t_1}
&\leq&\frac{2\epsilon C}{h_M^{-1}\Big(\frac{1}{2}|H(t_1)|^2+I_M-\epsilon C\Big)-\Omega_M}.
\end{eqnarray*}
\end{lemma}
\begin{proof}
Note that
\[\Omega(t)-\Omega_M=\Omega(t)-\Omega(t_1)=\int_{t_1}^{t}H(s)ds.\]
By Lemma \ref{lemma 5.6},
\[\Omega(\bar{t}_1)-\Omega_M\leq \int_{t_1}^{\bar{t}_1}|H(s)|ds\leq (\bar{t}_1-t_1)\sqrt{|H(t_1)|^2+4\epsilon C}.\]
Inequality \eqref{eq 5.111} in Lemma \ref{lemma 5.6} yields that
\[\frac{1}{2}|H(t_1)|^2+I_M-\epsilon C\leq h_M(\Omega(\bar{t}_1)).\]
Since $h_M$ is increasing and continuous,
\[h_M^{-1}\Big(\frac{1}{2}|H(t_1)|^2+I_M-\epsilon C\Big)\leq \Omega(\bar{t}_1).\]
Therefore,
\begin{eqnarray}\label{eq 4.10}
\frac{h_M^{-1}\Big(\frac{1}{2}|H(t_1)|^2+I_M-\epsilon C\Big)-\Omega_M}{\sqrt{|H(t_1)|^2+4\epsilon C}}\leq  \bar{t}_1-t_1.\end{eqnarray}

Similarly, we have
\[\Omega(\bar{t}_1)-\Omega_M\leq \int_{\bar{t}_1}^{t_1^r}|H(s)|ds\leq (t_1^r-\bar{t}_1)\sqrt{|H(t_1)|^2+4\epsilon C},\]
and
\[h_M^{-1}\Big(\frac{1}{2}|H(t_1)|^2+I_M-\epsilon C\Big)\leq \Omega(\bar{t}_1).\]

Therefore,
\begin{eqnarray}\label{eq 4.11}\frac{h_M^{-1}\Big(\frac{1}{2}|H(t_1)|^2+I_M-\epsilon C\Big)-\Omega_M}{\sqrt{|H(t_1)|^2+4\epsilon C}}\leq t^r_1- \bar{t}_1.\end{eqnarray}

By \eqref{eq 4.10} and \eqref{eq 4.11}, we have
\[2\frac{h_M^{-1}\Big(\frac{1}{2}|H(t_1)|^2+I_M-\epsilon C\Big)-\Omega_M}{\sqrt{|H(t_1)|^2+4\epsilon C}}\leq t^r_1- t_1.\]

By Lemma \ref{lemma 5.6},
\begin{eqnarray*}
\frac{\big||H(t_1^r)|-|H(t_1)|\big|}{t_1^r-t_1}\leq\frac{\sqrt{|H(t_1)|^2+4\epsilon C}-|H(t_1)|}{t_1^r-t_1}.
\end{eqnarray*}
Therefore, if  $h_M^{-1}\Big(\frac{1}{2}|H(t_1)|^2+I_M-\epsilon C\Big)-\Omega_M>0$, then we have
\begin{eqnarray*}
\frac{\big||H(t_1^r)|-|H(t_1)|\big|}{t_1^r-t_1}
&\leq&\frac{\sqrt{|H(t_1)|^2+4\epsilon C}-|H(t_1)|}{2}\frac{\sqrt{|H(t_1)|^2+4\epsilon C}}{h_M^{-1}\Big(\frac{1}{2}|H(t_1)|^2+I_M-\epsilon C\Big)-\Omega_M}\\
&\leq&\frac{4\epsilon C}{2(\sqrt{|H(t_1)|^2+4\epsilon C}+|H(t_1)|)}\frac{\sqrt{|H(t_1)|^2+4\epsilon C}}{h_M^{-1}\Big(\frac{1}{2}|H(t_1)|^2+I_M-\epsilon C\Big)-\Omega_M}
\\
&\leq&\frac{2\epsilon C}{h_M^{-1}\Big(\frac{1}{2}|H(t_1)|^2+I_M-\epsilon C\Big)-\Omega_M}
.\end{eqnarray*}

\end{proof}

\begin{lemma}

We assume that there is $C>0$ such that
\[\|F^-(t,\cdot,\cdot)\|_{L^\infty_{x,\omega}}< C,\quad t\geq 0. \]
Let $\Omega_m$  be the constant satisfying
\[F^h(\Omega_m)=C,\quad 0<\Omega_m<\frac{\epsilon}{2}<\epsilon. \]

We assume that
\[\Omega(t_1)=\Omega_m , \quad H(t_1)<0,\]
and  there is $t_1^*>0$ such that
\begin{eqnarray*}
\Omega(t)>\Omega_m,\quad \mbox{on $t\in (t_1^*,t_1)$.
}
\end{eqnarray*}
Then
\[2\frac{\Omega_m-h_m^{-1}\Big(\frac{1}{2}|H(t_1)|^2+I_m-\epsilon C\Big)}{\sqrt{|H(t_1)|^2+4\epsilon C}}\leq t^r_1- t_1,\]
where $t_1^r$ and $I_m$ are defined in Lemma \ref{lemma 5.6s}.

Moreover, if  $\Omega_m-h_m^{-1}\Big(\frac{1}{2}|H(t_1)|^2+I_m-\epsilon C\Big)>0$, then
\begin{eqnarray*}
\frac{\big||H(t_1^r)|-|H(t_1)|\big|}{t_1^r-t_1}
&\leq&\frac{2\epsilon C}{\Omega_m-h_m^{-1}\Big(\frac{1}{2}|H(t_1)|^2+I_m-\epsilon C\Big)}.
\end{eqnarray*}

\end{lemma}
\begin{proof}
The proof of this lemma is similar to Lemma \ref{lemma 5.7}.
\end{proof}

When the atom lies in the non-autonomous chaotic region $[\Omega_m, \Omega_M]$ where the two stretching forces are comparable, descriptive analysis as in the previous is not available. However,
the following rough bound is sufficient in this case:
\begin{lemma}\label{lemma 5.9} We assume that there is $C>0$ such that
\[\|F^-(t,\cdot,\cdot)\|_{L^\infty_{x,\omega}}<C,\quad t\geq 0. \]
Let  $\Omega_m$ and $\Omega_M$ be  constants satisfying
\[F^h(\Omega_m)=F^h(\Omega_M)=C,\quad 0<\Omega_m<\frac{\epsilon}{2}<\Omega_M<\epsilon. \]

If
\begin{eqnarray}\label{assumption}
\Omega_m<\Omega(t)<\Omega_M,\quad \mbox{on}~ t\in [t_1,t_2],
\end{eqnarray}
then\begin{eqnarray*}
|H(t)|\leq 2C (t-t_1)+|H(t_1)|,\quad \mbox{ on $t\in [t_1,t_2]$.
}
\end{eqnarray*}

\end{lemma}
\begin{proof}
Note that
\begin{eqnarray*}
 \frac{d}{ds}H(s) =
F^-(s,X(s),\Omega(s)) +F^h(\Omega(s)).
\end{eqnarray*}
Assumption \eqref{assumption} implies that
\begin{eqnarray*}
2|H(t)|\frac{d|H(t)|}{dt}=\frac{d}{dt}|H(t)|^2=2H(t)\frac{d H(t)}{dt}
& \leq &2|H(t)|~|F^-(s,X(s),\Omega(s)) +F^h(\Omega(s))|.
\end{eqnarray*}
Thus,
\begin{eqnarray*}
\frac{d|H(t)|}{dt}
& \leq &|F^-(s,X(s),\Omega(s)) +F^h(\Omega(s))|\leq |F^-(s,X(s),\Omega(s))| +|F^h(\Omega(s))|.
\end{eqnarray*}

On $t\in [t_1,t_2]$, the following holds:
\begin{eqnarray*}
\frac{d|H(t)|}{dt}
& \leq & |F^-(s,X(s),\Omega(s))| +|F^h(\Omega(s))|
\\
& \leq & C +C=2C
.
\end{eqnarray*}
This yields
\begin{eqnarray*}
|H(t)|\leq 2C (t-t_1)+|H(t_1)|.
\end{eqnarray*}

\end{proof}

\section{Existence of the global solution}\label{sec5}
\setcounter{equation}{0}
In this part, we  show that the continuation criterion in Theorem \ref{local} is satisfied under the assumptions of Theorem \ref{main result}, leading to the existence of the global solution to the main system.
We assume that initial data $\mathring{f}$ has compact support, i.e., there is a compact set $\mathcal{D}_c(0)=[-x_M,x_M]\times [-v_M,v_M]\times [\omega_m,\omega_M]\times [-\eta_M,\eta_M]\subset \bbr\times\bbr\times (0,\epsilon)\times \bbr$ such that
\[\mathring{f}(x,v,\omega,\eta)=0, \quad \mbox{for}\quad (x,v,\omega,\eta)  \notin \mathcal{D}_c(0). \]

By the continuation theorem, it suffices to prove that for a given time $T>0$, there is a compact set $\mathcal{D}_c(T)$ such that  the particle trajectory $[X(t), V(t),\Omega(t), H(t)]\in \mathcal{D}_c(T)$, for all  initial data $(x,v,\omega,\eta)\in \mathcal{D}_c$ and $t\in [0,T]$.

Note that by the same argument in Lemma \ref{lemma 4.4}, for given initial data $\mathring{f}$,
\begin{eqnarray*}\|F^{\pm}(t)\|_{L^\infty_{x,\omega}}\leq 2\|\mathring{f}\|_{L^1_{z}}.\end{eqnarray*}
Here we take a sufficiently large constant $C\gg \|\mathring{f}\|_{L^1_{z}}>0$ and constants $\Omega_m$, $\Omega_M$ such that
\[F^h(\Omega_m)=C,\quad F^h(\Omega_M)=-C,\quad 0<\Omega_m<\frac{\epsilon}{2}<\Omega_M<\epsilon, \]
and
\begin{eqnarray}\label{eq 6.2}
[\omega_m,\omega_M]\subset(\Omega_m,\Omega_M).
\end{eqnarray}

Consider the particle trajectory
 \[[X(s;t,z), V(s;t,z)]=[X(s;t,x,v,\omega,\eta), V(s;t,x,v,\omega,\eta)]\]
 and
 \[[\Omega(s;t,z), H(s;t,z)]=[\Omega(s;t,x,v,\omega,\eta), H(s;t,x,v,\omega,\eta)]\] passing through $(x,v)$ and $(\omega,\eta)$  at time $t=0$, respectively, satisfying  the following   ODE system.

\begin{align}
\begin{aligned}\nonumber 
\frac{d}{ds}X(s;0,z) &= V(s;0,z),\\
 \frac{d}{ds}V(s;0,z)& = F^+ (s,X(s;0,z),\Omega(s;0,z)),\\
\frac{d}{ds}\Omega(s;0,z) &= H(s;0,z),\\
 \frac{d}{ds}H(s;0,z) &=
F^-(s,X(s;0,z),\Omega(s;0,z)) +F^h(\Omega(s;0,z)),
\end{aligned}
\end{align}

with initial data
\begin{align}
\begin{aligned}\nonumber
X(0;0,z)& = x\in [-x_M,x_M], \\
V(0;0,z)& = v \in [-v_M,v_M],\\
\Omega(0;0,z)& = \omega\in [\omega_m,\omega_M]\subset (0,\epsilon),\\
 H(0;0,z) &= \eta \in [-\eta_M,\eta_M],
\end{aligned}
\end{align}
where  $F^\pm(t, \cdot,\cdot)$ is the self-consistence vector field satisfying \eqref{pVP}.
Clearly,
\[|X(t;0,z)|\leq x_M+v_M t+ \frac{C}{2}t^2\]
and
\[|V(t;0,z)|\leq v_M+ C t.\]

However, the cases for  $\Omega(t;0,z)$ and $H(t;0,z)$ are non-trivial. We define $S$ as
\[S=\{t>0: \Omega(t;0,z)=\Omega_m\quad \mbox{or}\quad  \Omega_M\}.\]
If $S$ is empty, then $\Omega(t;0,z)\in [\Omega_m,\Omega_M]$, for all $t>0$.
Then
\[\Omega_m\leq \Omega(t;0,z)\leq \Omega_M, \]
and
\begin{align}\label{f1}
|H(t;0,z)|\leq 2C t+\eta_M.
\end{align}
Therefore, we may assume that $S$ is nonempty.  From the continuity of the particle trajectories,  $S$ consists of countable points and closed intervals.

Let $\{t_n,~n\in \bbn\}$ be the set of  isolated points in $S$ and the end points of closed intervals in $S$. We may assume that
 \[t_1<t_2<\ldots<t_i<\ldots~.\] Note that the particle trajectory is continuous. Thus, by  \eqref{eq 6.2}, there is a $t_1>0$ such that
\[\Omega(t_1;0,z)=\Omega_m\quad \mbox{or}\quad  \Omega_M.\]
Without loss of generality, we assume that
\[
\Omega(t_1;0,z)= \Omega_M,
\]
since the case $\Omega(t;0,z)=\Omega_m$ can be treated identically.

This assumption implies that $\Omega(t;0,z)\in [\Omega_m,\Omega_M]$ on  $t\in [0,t_1]$.
By Lemma \ref{lemma 5.9},
\begin{eqnarray*}
|H(t)|\leq 2C t+|H(0)|\leq  2C t+\eta_M,\quad \mbox{on  $t\in [0,t_1]$.}
\end{eqnarray*}

By Lemmas \ref{lemma 5.4} and \ref{lemma 5.6}, there is $t_2>0$ such that $t_2>t_1$ and $t_2\in S$, i.e., \[\Omega(t_2;0,z)=\Omega(t_1;0,z)=\Omega_M.\]

Since $\Omega(t;0,z)\in (\Omega_M,\epsilon)$ on  $t\in [t_1,t_2]$,  if  $h_M^{-1}\Big(\frac{1}{2}|H(t_1)|^2+I_M-\epsilon C\Big)-\Omega_M>0$, then
\begin{eqnarray}\label{eq 6.3}
\frac{\big||H(t_1^r)|-|H(t_1)|\big|}{t_1^r-t_1}
&\leq&\frac{2\epsilon C}{h_M^{-1}\Big(\frac{1}{2}|H(t_1)|^2+I_M-\epsilon C\Big)-\Omega_M}.
\end{eqnarray}

We consider the following two cases.

\begin{itemize}
\item[$\circ$] $|H(t_1)|^2<4\epsilon C$: For this case, we can apply Lemma \ref{lemma 5.6} to obtain
\begin{align}\label{f2}
|H(t)|
\leq \big(|H(t_1)|^2+4\epsilon C\big)^{\frac{1}{2}}\leq  (8\epsilon C)^{\frac{1}{2}},\quad \mbox{for all $t\in [t_1,t_2]$.
}
\end{align}

\item[$\circ$] $|H(t_1)|^2\geq 4\epsilon C$: Since $h_M$ is monotone increasing, we have
\begin{align}\label{f3}
h_M^{-1}\Big(\frac{1}{2}|H(t_1)|^2+I_M-\epsilon C\Big)-\Omega_M\geq h_M^{-1}\Big(\epsilon C+I_M\Big)-\Omega_M>0.
\end{align}

 \eqref{eq 6.3} and \eqref{f3} yield that
\begin{eqnarray*}
|H(t)|
\leq |H(t_1)|+ C_1(t_2-t_1),\quad \mbox{for all $t\in [t_1,t_2]$,}
\end{eqnarray*}
where
\[C_1=\frac{2\epsilon C}{h_M^{-1}\Big(\epsilon C+I_M\Big)-\Omega_M}.\]

\end{itemize}
We can apply this argument for $t_3,t_4,\ldots$~.

Finally, we conclude from (\ref{f1}), (\ref{f2}) and (\ref{f3}) that on $0<t<T$,
\[|H(t;0,z)|\leq \sqrt{(C_2T+\eta_M)^2+4\epsilon C},\]
where $C_2=\max\{2C,C_1\}$.

Now we use Lemma \ref{lemma 5.1}: for $0<t<T$,
\begin{eqnarray*}
\int^{\frac{\epsilon}{2}}_{\Omega(t)}F^h(y)dy
& =&
\int_{0}^{t} H(s)F^-(s,X(s),\Omega(s))ds +\int_{\Omega(0)}^{\frac{\epsilon}{2}}F^h(y)dy+\frac{1}{2}|H(0)|^2-\frac{1}{2}|H(t)|^2\\
& \leq &
\int_{0}^{t} |H(s)|~|F^-(s,X(s),\Omega(s))|ds +\int_{\Omega(0)}^{\frac{\epsilon}{2}}F^h(y)dy+\frac{1}{2}|H(0)|^2\\
& \leq &
C\int_{0}^{T}\sqrt{(C_2T+\eta_M)^2+4\epsilon C} ds +\int_{\Omega(0)}^{\frac{\epsilon}{2}}F^h(y)dy+\frac{1}{2}|H(0)|^2
\\
& \leq &
CT\sqrt{(C_2T+\eta_M)^2+4\epsilon C} +\int_{\Omega(0)}^{\frac{\epsilon}{2}}F^h(y)dy+\frac{1}{2}|H(0)|^2.
\end{eqnarray*}
Thus, there are $\Omega^T_m$, $\Omega^T_M\in \bbr$ such that
for $0< t< T$,
\[0<\Omega^T_m\leq \Omega(t;0,z)\leq \Omega^T_M<\epsilon.\]
Thus, we can conclude that there is the unique global solution to \eqref{pVP} by the continuation criterion in Theorem \ref{local}.

\bibliographystyle{amsplain}

\begin{thebibliography}{10}




\bibitem{A-A}Ambrosino, G. and  Albanese, R.: Magnetic control of plasma current, position, and shape in Tokamaks: a survey or modeling and control approaches. IEEE Control Systems Magazine, 25 (2005) 76--92.

\bibitem{B-D}Bardos, C. and Degond, P.: Global existence for the Vlasov-Poisson equation in 3 space variables with small initial data. Ann. Inst. H. Poincaré Anal. Non Linéaire, 2 (1985) 101--118.

\bibitem{Batt}Batt, J.: Global symmetric solutions of the initial value problem in stellar dynamics, J. Differential Equations, 25  (1977) 342--364.


\bibitem{B-C-H}Bird, A. R. B,  Curtiss, C. F. and Hassager, O.: Dynamics of polymeric fluids. Kinetic theory,
volume 2. 2nd edition, John Wiley, New York, 1987.
\bibitem{Bryan} Bryan, G. H. : On the application of the determinantal relation to the kinetic theory of polyatomic gases,R ept. Brit. Assoc. Adv. Sci., 64 (1894) 102--106.



\bibitem{CC} Chapman, S. and Cowling, T. G.: The Mathematical Theory of Non-uniform Gases. Cambridge University Press, Cambridge, 1939. xxiii+404pp.
\bibitem{D-L}Degond, P. and Liu, H.:  Kinetic models for polymers with inertial effects. Networks and  Heterogeneous Media, 4.4 (2009) 625--647.


\bibitem{Guo} Guo, Y.: Regularity for the Vlasov equations in a half-space, Indiana University Mathematics Journal,  43 (1994) 255--321.

\bibitem{H-H-N}Horst, E.,  Hunze, R. and Neunzert, H.: Weak solutions of the initial value problem for the unmodified nonlinear Vlasov equation,  Mathematical Methods in the Applied Sciences,  6  (1984) 262--279.



 \bibitem{Hwang} Hwang, H. J.: Regularity for the Vlasov-Poisson System in a Convex Domain, SIAM journal on mathematical analysis, 36 (2004) 121--171 .


\bibitem{H-V1} Hwang, H. J. and  Velázquez J.: Global existence for the Vlasov-Poisson system in bounded domains, Archive for rational mechanics and analysis, 195  (2010) 763--796.

\bibitem{H-V2}Hwang, H. J. and  Velázquez J.: On global existence for the Vlasov-Poisson system in a half space,  Journal of Differential Equations, 247  (2009) 1915--1948.



\bibitem{K-R}Köhler, W. E. and  Raum, H. H.: Kinetic Theory for Mixtures of Dilute Gases of Linear Rotating Molecules II. The Senftleben-Beenakker Effect of the Heat Conductivity. Zeitschrift für Naturforschung A, 27 (1972) 1383--1393.






\bibitem{Lions-Perthame}
 Lions, P.-L. and Perthame, B.: Propagation of moments and regularity for the 3-dimensional Vlasov-Poisson system. Invent. Math., 105 (1991) 415--430.



\bibitem{Pfaffelmoser} Pfaffelmoser, K.: Global classical solutions of the Vlasov-Poisson system in three dimensions for general initial data. J. Differential Equations, 95 (1992) 281--303.

    \bibitem{Rein}Rein, G.: Collisionless kinetic equations from astrophysics the Vlasov-Poisson system. Handbook of differential equations: evolutionary equations, 3 (2007) 383--476.



\bibitem{T-C-T-B}Turfa, A. F., Connor, J. N. L., Thijsse, B. J. and  Beenakker, J. J. M.: A classical dynamics study of Senftleben-Beenakker effects in nitrogen gas. Physica A: Statistical Mechanics and its Applications, 129  (1985) 439--454.















\bibitem{U-O}Ukai, S., and Okabe, T. : On classical solutions in the large in time of two-dimensional Vlasov's equation, Osaka Journal of Mathematics, 15  (1978) 245--261.



\end{thebibliography}

\end{document}